\documentclass[fleqn,final]{zzz}
\usepackage{mathrsfs}
\usepackage{color}
\usepackage{tikz}
\usepackage{amsmath,amsthm,amsbsy,amssymb}
\usepackage{graphicx}
\usepackage{rotating}
\usepackage{fixmath}
\usepackage[pdftex]{hyperref}
\usepackage{soul}
\usepackage{comment}
\usepackage{cancel}
\usepackage{bm}

\hypersetup{
 colorlinks = true,
 urlcolor = blue,
 linkcolor = red,
 citecolor = green,
}

\sethlcolor{black}

\newcommand{\Whyp}[5]{\,\mbox{}_{#1}W_{#2}\!\left({#3};{#4};{#5}\right)}
\newcommand{\qpWhyp}[6]{\,\sideset{_{#1}^{\phantom{\mid}}}{_{#2}^{#3}}
{\mathop{W}}\!\left({#4};{#5};#6\right)}
\newcommand{\qhyp}[5]{\,\mbox{}_{#1}\phi_{#2}\!\left(
\genfrac{}{}{0pt}{}{#3}{#4};#5\right)}
\newcommand{\qphyp}[6]{\,\sideset{_{#1}^{\phantom{\mid}}}{_{#2}^{#3}}
{\mathop{\phi}}\!\left(\genfrac{}{}{0pt}{}{#4}{#5};#6\right)}

\newcommand{\rphis}[2]{{}_{#1\vphantom{#2}}\phi_{#2\vphantom{#1}}}
\newcommand{\rphisx}[4]{\rphis{#1}{#2}\left(\begin{array}{c} #3\end{array};q,#4\right)}

\newtheorem{thm}{Theorem}[section]
\newtheorem{cor}[thm]{Corollary}

\newtheorem{rem}[thm]{Remark}
\newtheorem{lem}[thm]{Lemma}
\newtheorem{defn}[thm]{Definition}
\newtheorem{prop}[thm]{Proposition}

\makeatletter
\def\eqnarray{\stepcounter{equation}\let\@currentlabel=\theequation
\global\@eqnswtrue
\tabskip\@centering\let\\=\@eqncr
$$\halign to \displaywidth\bgroup\hfil\global\@eqcnt\z@
$\displaystyle\tabskip\z@{##}$&\global\@eqcnt\@ne
\hfil$\displaystyle{{}##{}}$\hfil
&\global\@eqcnt\tw@ $\displaystyle{##}$\hfil
\tabskip\@centering&\llap{##}\tabskip\z@\cr}

\def\endeqnarray{\@@eqncr\egroup
\global\advance\c@equation\m@ne$$\global\@ignoretrue}

\def\@yeqncr{\@ifnextchar [{\@xeqncr}{\@xeqncr[5pt]}}
\makeatother

\newcommand{\Z}{\mathbb{Z}} 
\newcommand{\C}{\mathbb{C}} 
\newcommand{\N}{\mathbb{N}} 

\newcommand{\CC}{{{\mathbb C}}}
\newcommand{\CCast}{{{\mathbb C}^\ast}}
\newcommand{\CCdag}{{{\mathbb C}^\dag}}
\newcommand{\CCddag}{{{\mathbb C}^\ddag}}

\newcommand{\midtilde}{\raisebox{-0.25\baselineskip}{\textasciitilde}}

\usepackage{scalerel,url}
\let\svus_
\catcode`_=\active %
\def_{\ifmmode\expandafter\svus\expandafter\bgroup\expandafter\lowerit
\else\svus\fi}
\def\lowerit#1{\ThisStyle{\raisebox{-2\LMpt}{$\SavedStyle#1$}}\egroup}
\makeatletter
 \AtBeginDocument{
 \check@mathfonts
 \fontdimen13\textfont2=3.5pt
 \fontdimen14\textfont2=3.5pt
 \fontdimen16\textfont2=2.5pt
 \fontdimen17\textfont2=2.5pt
 }
\makeatother

\begin{document}

\renewcommand{\PaperNumber}{***}

\FirstPageHeading

\ShortArticleName{Connection formulas for Askey--Wilson polynomials and related expansions}

\ArticleName{Connection formulas for Askey--Wilson polynomials\\and related expansions}

\Author{Howard S. Cohl$\,^{\ast}$ and Wolter Groenevelt$\,^{\dag}$ 
}

\AuthorNameForHeading{H.~S.~Cohl}
\Address{$^\ast$ Applied and Computational 
Mathematics Division, National Institute of Standards 
and Tech\-no\-lo\-gy, Mission Viejo, CA 92694, USA
\URLaddressD{
\href{http://www.nist.gov/itl/math/msg/howard-s-cohl.cfm}
{http://www.nist.gov/itl/math/msg/howard-s-cohl.cfm}
}
} 
\EmailD{howard.cohl@nist.gov} 

\AuthorNameForHeading{H.~S.~Cohl, W.~Groenevelt}
\Address{$^\dag$ Delft Institute of Applied Mathematics, Delft University of Technology, P.O. Box 5031, 2600
GA Delft, The Netherlands
\URLaddressD{
\href{https://fa.ewi.tudelft.nl/~groenevelt/}
{https://fa.ewi.tudelft.nl/\midtilde{}groenevelt/}
}
} 
\EmailD{w.g.m.groenevelt@tudelft.nl} 


\ArticleDates{Received \today~in final form ????; Published online ????}

\Abstract{%
We derive and study expansions of and over the Askey--Wilson polynomials. We study these expansions and examine some limits to the continuous dual $q$-Hahn, Al-Salam--Chihara, continuous big $q$-Hermite and continuous $q$-Hermite polynomials and their $q^{-1}$-analogues. The Poisson kernel for the infinite discrete orthogonality relation for the $q^{-1}$-Al-Salam--Chihara polynomials is derived which in a special case reduces to  the Gupta--Masson biorthogonal rational ${}_4\phi_3$-functions. This Poisson kernel implies new infinite series connection relations for the Askey--Wilson polynomials involving these rational ${}_4\phi_3$-functions. We also consider various interesting limits.
}

\Classification{33D15; 33D50}

\begin{flushright}
\begin{minipage}{60mm}
\it Dedicated to Krishna Alladi\\on the occasion of his 70\textsuperscript{th} birthday.
\end{minipage}
\end{flushright}

\section{Introduction}


The Askey--Wilson polynomials \cite{AskeyWilson85} form a family of orthogonal polynomials that are on the top-level of the Askey-scheme of $q$-hypergeometric orthogonal polynomials, implying that they contain many well-known families of orthogonal polynomials as special cases or limit cases. This makes them a central object in the study of special functions. This paper contributes to this study by considering several summation formulas involving Askey--Wilson polynomials and some of their limit cases. 

Askey--Wilson polynomials are defined explicitly as
\[
p_n(x;a,b,c,d|q) = a^{-n} (ab,ac,ad;q)_n \qhyp43{q^{-n},q^{n-1}abcd, az, az^{-1}}{ab,ac,ad}{q,q}, \qquad x = \frac12(z+z^{-1}),
\]
where we use the usual notations for $q$-shifted factorials and $q$-hypergeometric series. Clearly, the polynomials $p_n$ depend, beside the `base' parameter $q$, on the four parameters $a,b,c,d$. A fundamental result for Askey--Wilson polynomials is the connection formula
\[
p_n(x;a,b,c,d|q) = \sum_{k=0}^n c_{k,n}\,  p_n(x;e,f,g,h|q),
\]
relating the family of polynomials depending on $a,b,c,d$ to the polynomials depending on different parameters $e,f,g,h$.
Askey and Wilson computed \cite{AskeyWilson85} the connection coefficients $c_{k,n}$ in case $d=h$ as balanced $_5\phi_4$-series. Ismail and Zhang \cite[Theorem 3.2]{IsmailZhang2005}
obtained a formula for connection coefficients in the general case as a double sum. In this paper we obtain an explicit expression as balanced $_4\phi_3$-functions for the coefficients in the slightly adjusted connection problem
\begin{equation} \label{connection coeff}
\frac{ p_n(x;a,b,c,d|q)}{(cz,cz^{-1};q)_\infty} = \sum_{k=0}^\infty C_{k,n} \,\frac{p_n(x;a,b,e,f|q)}{(ez,ez^{-1};q)_\infty}.
\end{equation}

To compute the coefficients $C_{k,n}$ we first evaluate a Poisson kernel for Al-Salam--Chihara polynomials in base $q^{-1}$. Then using a summation formula which relates $q^{-1}$-Al-Salam--Chihara polynomials to Askey--Wilson polynomials it follows that a special case of the Poisson kernel gives the connection coefficients. Furthermore, from orthogonality relations for $q^{-1}$-Al-Salam--Chihara polynomials it follows that the coefficients $C_{k,n}$ can be considered as biorthogonal functions in $k$ (or $n$). This leads to another proof of biorthogonality relations for $_4\phi_3$-functions obtained by Gupta and Masson \cite{GuptaMasson1998}. 

The motivation for looking at these summations comes from their appearance in the representation theory of the  quantum algebra $\mathcal U_q(\mathfrak{sl}_{2})$, e.g., see \cite{Groenevelt2021} and \cite{GroeneveltWagenaar25}. In \cite{Groenevelt2021} the Askey--Wilson polynomials, including the denominator factors in \eqref{connection coeff}, appear as overlap coefficients between bases $\{v_n\}$ and $\{\tilde v_n\}$ 
of the representation space. These two bases diagonalize two different operators in $\mathcal U_q(\mathfrak{sl}_{2})$ and can be written in terms of Al-Salam--Chihara polynomials in base $q$ and $q^{-1}$.
This goes back to Rosengren's work \cite{RosengrenCONM2000} on generalized group elements in $\mathcal U_q(\mathfrak{sl}_{2})$. 
The connection formula \eqref{connection coeff} can be interpreted in this setting as a result of computing overlap coefficients in two different ways: directly using the bases $\{v_n\}$ and $\{\tilde v_n\}$, or in two steps by first computing overlap coefficient between $\{v_n\}$ and a natural `in between' basis $\{\tilde v_n'\}$ and then between $\{\tilde v_n'\}$ and $\{\tilde v_n\}$. 
Computing overlap coefficients between $\{\tilde v_n'\}$ and $\{\tilde v_n\}$ requires an evaluation of the Poisson kernel mentioned above. We stress that the approach in the current paper does not rely anywhere on knowledge of the algebraic interpretations. 


\medskip

The paper is organized as follows. In this paper we consider expansions using Askey--Wilson polynomials. In Section \ref{Preliminaries} we introduce various mathematical concepts which are required in order to read the paper. In Section \ref{NonstanAWgf} we consider certain interesting limits and specializations of the nonstandard generating function for Askey--Wilson polynomials which was originally published in \cite[Theorem 2.1]{Groenevelt2004}. In Section \ref{Poissonker} we consider a Poisson kernel for the $q^{-1}$-Al-Salam--Chihara polynomials and its limit to a Poisson kernel for the continuous big $q^{-1}$-Hermite polynomials. In Section \ref{binlinsec}, we utilize a bilinear sum over a product of Al-Salam--Chihara and $q^{-1}$-Al-Salam--Chihara polynomials originally presented in \cite{Groenevelt2021} along with an orthogonality relation for $q^{-1}$-Al-Salam--Chihara polynomials in order to obtain in Lemma \ref{lemASCAWQi}, a bilinear sum of the Al-Salam--Chihara polynomials in terms of a product of Askey--Wilson polynomials and $q^{-1}$-Al-Salam--Chihara polynomials. We also compute some interesting limits of Lemma \ref{lemASCAWQi}. In Section \ref{genexpan} we utilize Lemma \ref{lemASCAWQi} in order to derive two connection relations for the Askey--Wilson polynomials in Theorem \ref{Thm:sum AWpol} whose coefficients are given in terms ${}_4\phi_3$-functions. We also consider interesting limits of these connection relations. 
In Section \ref{GuptaMasson} we focus on the relation between the connection coefficients and Gupta and Masson's biorthogonal rational ${}_4\phi_3$-functions. Among other things, we obtain in Corollary \ref{cor72} an integral representation for the rational $_4\phi_3$-functions which generalizes the well-known Askey--Wilson orthogonality relations.

\section{Preliminaries}
\label{Preliminaries}
{
We adopt the following set 
notations:~$\N_0:=\{0\}\cup\N=\{0, 1, 2, ...\}$, and we use the set $\CC$ which represents the complex numbers, 
$\CCast:=\CC\setminus\{0\}$,  
$\CCdag:=\{z\in\CCast: |z|<1\}$,\
$\CCddag:=\CC\setminus(\{0\}\cup\{z\in\CCast:|z|=1\})$.

\begin{defn} \label{def:2.1}
Throughout this paper we adopt the following conventions for succinctly 
writing elements of lists. To indicate sequential positive and negative 
elements, we write
\[
\pm a:=\{a,-a\}.
\]
\noindent We also adopt an analogous multiset notation
\[
z^{\pm}:=z^{\pm 1}:=\{ z,z^{-1}\}.
\]
\end{defn}
\noindent Consider $q\in\CCdag$.
Define for $n\in\N_0$ the sets 
$\Omega_q^n:=\{q^{-k}:k\in\N_0,~0\le k\le n-1\}$,
$\Omega_q:=\Omega_q^\infty=\{q^{-k}:k\in\N_0\}$.
In order to obtain our derived identities, we rely on properties 
of the $q$-Pochhammer symbol (a.k.a.\ the $q$-shifted factorial). 
For any $n\in \N_0$, 
$q\in\CCdag$, $a,b, \in \CC$, 
the 
$q$-Pochhammer symbol is defined as
\[
(a;q)_n:=(1-a)(1-aq)\cdots(1-aq^{n-1}),\quad n\in\N_0,
   \quad\text{and }(a;q)_\infty=\lim_{n\to\infty}(a;q)_n.
\]
We will also use, for $k\in\N$ and $n\in\N_0\cup\{\infty\}$,
the common notational product convention
\[
(a_1,...,a_k;q)_n:=(a_1;q)_n\cdots(a_k;q)_n.
\]
For $n\in\CC$, one has \cite[(I.5)]{GaspRah}
\begin{equation}
(q^na;q)_\infty=\frac{(a;q)_\infty}{(a;q)_n}.
\label{inffinrel}
\end{equation}
Basic hypergeometric series, which we 
will often use, are defined for
$q\in\CCdag$, $z\in\CCast$, such that $|q|<1$, $s,r\in\N_0$, 
$b_j\not\in\Omega_q$, $j=1,...,s$, as
\cite[(1.10.1)]{Koekoeketal}
\begin{equation}
\qhyp{r}{s}{a_1,...,a_r}
{b_1,...,b_s}
{q,z}
:=\sum_{k=0}^\infty
\frac{(a_1,...,a_r;q)_k}
{(q,b_1,...,b_s;q)_k}
\left((-1)^kq^{\binom k2}\right)^{1+s-r}
z^k.
\label{2.11}
\end{equation}
where for $r=s+1$ we assume $|z|<1$ for convergence, if the series
does not terminate. (For $r\le s$ the series converges for any $z$,
due to the quadratic powers of $q$ appearing in the terms of the series.)
We refer to a basic hypergeometric series as balanced 
(or Saalsch\"utzian) if $r=s+1$ and $q a_1\cdots a_r=b_1\cdots b_s$.
In the sequel, we will use the following notation 
${}_{r+1}\phi_s^m$, $m\in\mathbb Z$
(originally due to van de Bult \& Rains
\cite[p.~4]{vandeBultRains09}), 
for basic hypergeometric series with
zero parameter entries.
Consider $p\in\mathbb N_0$. Then define
\begin{equation}\label{topzero} 
{}_{r+1}\phi_s^{-p}\left(\begin{array}{c}a_1,\ldots,a_{r+1}\\
b_1,\ldots,b_s\end{array};q,z
\right)
:=
\rphisx{r+p+1}{s}
{a_1,a_2,\ldots,a_{r+1},\overbrace{0,\ldots,0}^{p}\\ 
b_1,b_2,\ldots,b_s}{z},
\end{equation}
\begin{equation}\label{botzero}
{}_{r+1}\phi_s^{\,p}\left(\begin{array}{c}a_1,\ldots,a_{r+1}\\
b_1,\ldots,b_s\end{array};q,z
\right)
:=
\rphisx{r+1}{s+p}{a_1,a_2,\ldots,a_{r+1}\\ 
b_1,b_2,\ldots,b_s,\underbrace{0,\ldots,0}_p}{z},
\end{equation}
where $b_1,\ldots,b_s\not
\in\Omega_q\cup\{0\}$, and
${}_{r+1}\phi_s^0:={}_{r+1}\phi_{s}$.
The nonterminating basic hypergeometric series 
${}_{r+1}\phi_s^m({\bf a};{\bf b};q,z)$, 
${\bf a}:=\{a_1,\ldots,a_{r+1}\}$,
${\bf b}:=\{b_1,\ldots,b_s\}$, is well-defined for 
$s-r+m\ge 0$. 
In particular ${}_{r+1}\phi_s^m$ is an entire function 
of $z$ for $s-r+m>0$, convergent for $|z|<1$ for $s-r+m=0$ 
and divergent if $s-r+m<0$.
Note that we will move interchangeably between the
van de Bult \& Rains notation and the alternative
notation with vanishing numerator and denominator parameters
which are used on the right-hand sides of \eqref{topzero} 
and \eqref{botzero}.

One has the following nonterminating transformation between a ${}_2\phi_2$ 
and a ${}_2\phi_1$ cf.~ \cite[(III.4)]{GaspRah}
\begin{equation}
\qhyp22{a,b}{c,\frac{abz}{c}}{q,z}=
\frac{(\frac{bz}{c};q)_\infty}{(\frac{abz}{c};q)_\infty}
\qhyp21{a,\frac{c}{b}}{c}{q,\frac{bz}{c}}.
\label{rel2122}
\end{equation}
\noindent Another important nonterminating transformation that we will use is 
\cite[\href{http://dlmf.nist.gov/17.9.E13}{(17.9.13)}]{NIST:DLMF}
\begin{equation}
\qhyp32{a,b,c}{d,e}{q,\frac{de}{abc}}=\frac{(\frac{e}{b},\frac{e}{c};q)_\infty}{(e,\frac{e}{bc};q)_\infty}\qhyp32{\frac{d}{a},b,c}{d,\frac{qbc}{e}}{q,q}+\frac{(\frac{d}{a},b,c,\frac{de}{bc};q)_\infty}
{(d,e,\frac{bc}{e},\frac{de}{abc};q)_\infty}
\qhyp32{\frac{e}{b},\frac{e}{c},\frac{de}{abc}}{\frac{de}{bc},\frac{qe}{bc}}{q,q}.
\label{3phi2nbtran}
\end{equation}
Another important nonterminating 
transformation formula that we will use is
Bailey's transformation of a nonterminating very-well-poised ${}_8W_7$ to a sum of two nonterminating balanced ${}_4\phi_3$'s
\cite[\href{http://dlmf.nist.gov/17.9.E16}{(17.9.16)}]{NIST:DLMF}
\vspace{12pt}
\begin{eqnarray}
&&\hspace{-0.9cm}{}_8W_7\left(a;b,c,d,e,f;q,\frac{q^2a^2}{bcdef}\right)
=\frac{(qa,\frac{qa}{de},\frac{qa}{df},\frac{qa}{ef};q)_\infty}
{(\frac{qa}{def},\frac{qa}{d},\frac{qa}{e},\frac{qa}{f};q)_\infty}
\qhyp43{\frac{qa}{bc},d,e,f}{\frac{qa}{b},\frac{qa}{c},\frac{def}{a}}
{q,q}\nonumber\\&&
\hspace{3cm}+\frac{(qa,\frac{q^2a^2}{bdef},\frac{q^2a^2}{cdef},\frac{qa}{bc},d,e,f
;q)_\infty}
{(\frac{q^2a^2}{bcdef},\frac{def}{qa},\frac{qa}{b},\frac{qa}{c},\frac{qa}{d},\frac{qa}{e},\frac{qa}{f};q)_\infty}
\qhyp43{\frac{q^2a^2}{bcdef},\frac{qa}{de},\frac{qa}{df},\frac{qa}{ef}}
{\frac{q^2a^2}{bdef},\frac{q^2a^2}{cdef},\frac{q^2a}{def}}{q,q},
\label{Baileytran}
\end{eqnarray}
where $\frac{qa}{b},\frac{qa}{c},\frac{qa}{d},\frac{qa}{e},\frac{qa}{f},\frac{qa}{def},\frac{def}{qa},\frac{q^2a^2}{bcdef}\not\in\Upsilon_q$.
Let $p\in\Z$.
Consider the nonterminating very-well poised 
basic hypergeometric series with vanishing denominator parameters
${}_{r+1}W_r^p$ \cite[p.~4]{vandeBultRains09}
\begin{equation}
\label{rpWr}
{}_{r+1}W_r^p(a;a_4,\ldots,a_{r+1};q,z)
:=\qphyp{r+1}{r}{p}{a,\pm q\sqrt{a},a_4,\ldots,a_{r+1}}
{\pm \sqrt{a},\frac{qa}{a_4},\ldots,\frac{qa}{a_{r+1}}}{q,z},
\end{equation} 
where $\sqrt{a},\frac{qa}{a_4},\ldots,\frac{qa}{a_{r+1}}\not\in\Omega_q$. Recognize that ${}_{r+1}W_r^0={}_{r+1}W_r$. Also note that in the case where $p<0$, then the nonterminating basic hypergeometric series is divergent, so we will not consider this case. As pointed out in \cite[p.~4]{vandeBultRains09}, usually one cannot obtain the nonterminating very-well-poised series with vanishing denominator parameters as a limit from very-well-poised series. However in cases we will treat below, this limit will be well-defined. This fact leads to some interesting nonterminating transformation formulas for these functions.
\begin{thm}
Let $q\in\CCdag$, $a,b,c,d,e\in\CCast$. Then
\begin{eqnarray}
\label{W761tran}
&&\hspace{-0.6cm}\qpWhyp{7}{6}{1}{a}{b,c,d,e}{q,\frac{q^2a^2}{bcde}}=
\frac{(qa,\frac{qa}{de};q)_\infty}{(\frac{qa}{d},\frac{qa}{e};q)_\infty}\qhyp32{\frac{qa}{bc},d,e}{\frac{qa}{b},\frac{qa}{c}}{q,\frac{qa}{de}},
\\
\label{W652tran}
&&\hspace{-0.6cm}\qpWhyp{6}{5}{2}{a}{b,c,d}{q,\frac{q^2a^2}{bcd}}=
\frac{(qa;q)_\infty}{(\frac{qa}{d};q)_\infty}\qhyp22{\frac{qa}{bc},d}{\frac{qa}{b},\frac{qa}{c}}{q,\frac{qa}{d}}=
\frac{(qa,\frac{qa}{bd};q)_\infty}{(\frac{qa}{b},\frac{qa}{d};q)_\infty}\qhyp21{b,d}{\frac{qa}{c}}{q,\frac{qa}{bd}}.
\end{eqnarray}
\end{thm}
\begin{proof}
For \eqref{W761tran}, start with Bailey's transformation of a nonterminating very-well-poised ${}_8W_7$ to a sum of two nonterminating balanced ${}_4\phi_3$'s
\eqref{Baileytran}
and take the limit $b\to\infty$ (or $c\to\infty$). Then replace $f\mapsto b$ (or $f\mapsto c$),
and apply the three-term nonterminating transformation for a ${}_3\phi_2$
\eqref{3phi2nbtran}.
For \eqref{W652tran}, start with \eqref{W761tran},
take the limit as $e\to\infty$ (or $d\to\infty$) 
produces the ${}_2\phi_2$ representation.  Taking the limit as $c\to\infty$ (or $b\to\infty$) and then 
replacing variables accordingly so that $b,c,d$ remains, produces
the result.
\end{proof}

\noindent Two useful terminating summations that we will use are the terminating $q$-binomial theorem, namely \cite[(1.11.2)]{Koekoeketal} 
\begin{equation}
\qhyp10{q^{-n}}{-}{q,z}=(q^{-n}z;q)_n, 
\label{termbinom}
\end{equation}
and the $q$-Chu--Vandermonde sum \cite[\href{http://dlmf.nist.gov/17.6.E3}{(17.6.3)}]{NIST:DLMF}
\begin{eqnarray}
&&\hspace{-10.3cm}\qhyp21{q^{-n},a}{b}{q,q}=\frac{a^n(\frac{b}{a};q)_n}{(b;q)_n}.
\label{qChuVander}
\end{eqnarray}

\subsection{The Askey--Wilson polynomials and their symmetric subfamilies}
Below we present definitions for orthogonal polynomials we are using here.
All basic hypergeometric orthogonal polynomials which we use are adopted from the definitions in Koekoek, Lesky \& Swarttouw (2010) \cite{Koekoeketal}.
The Askey--Wilson polynomials~\cite[Chapter~15]{Ismail} form the most general
family of orthogonal polynomials in $x$ that are also basic hypergeometric
series. They contain, in addition to the variable $x$ and the base $q$,
four extra parameters $a,b,c,d$. The $n$-th degree Askey--Wilson
polynomial is denoted by $p_n(x;a,b,c,d|q)$.
These polynomials are symmetric in the four
variables $a,b,c,d$ and have the following
terminating balanced ${}_4\phi_3$ basic hypergeometric series
representation
\cite[(15)]{CohlCostasSantos20b}
(with obvious additional ones corresponding to the symmetries
in the four parameters).
Let $n\in\N_0$, $q\in\CCddag$,
$x=\frac12(z+z^{-1})$, $z,a,b,c,d\in\CCast$.
Then the Askey--Wilson polynomials can be defined as follows \cite[Theorem~7]{CohlCostasSantos20b}
\begin{eqnarray}
\label{aw:def1} 
&&\hspace{-3cm}p_n(x;a,b,c,d|q)=a^{-n} (ab,ac,ad;q)_n 
\qhyp43{q^{-n},q^{n-1}abcd, az^{\pm}}{ab,ac,ad}{q,q}.
\label{aw:def2} 
\end{eqnarray}
Using Sears' transformations \cite[(III.15), (III.16)]{GaspRah} several other $_4\phi_3$-expressions for the Askey--Wilson polynomials can be obtained.

Setting parameters equal to 0 gives several subfamilies of the Askey--Wilson polynomials: 
The continuous dual $q$-Hahn polynomials can be given by 
\begin{eqnarray}
\label{cdqH:def1} &&\hspace{-5.35cm}p_n(x;a,b,c|q) 
:= a^{-n} (ab,ac;q)_n 
\qhyp{3}{2}{q^{-n}, az^{\pm}}{ab,ac}{q,q}
\end{eqnarray}

\noindent The Al-Salam--Chihara polynomials are given by 
\begin{eqnarray}
\label{ASC:def1} 
&&\hspace{-5.8cm}Q_n(x;a,b|q):=a^{-n}(ab;q)_n
\qphyp{3}{1}{1}{q^{-n}, az^{\pm}}{ab}{q,q}\\
\label{ASC:def4} &&\hspace{-3.6cm}=z^{n}(b z^{-1};q)_n 
\qhyp{2}{1}{q^{-n},az}
{\frac{q^{1-n}}{b}z}{q,\frac{q}{bz}}.
\end{eqnarray}

\noindent The continuous big $q$-Hermite polynomials are given by 
\begin{eqnarray}
\label{cbqH:def1}
&&\hspace{-7.2cm}H_n(x;a|q)
:=a^{-n}\qphyp{3}{0}{2}{q^{-n}, az^\pm }{-}{q,q}\\
\label{cbqH:def4}
&&\hspace{-5.3cm}=z^n\qhyp{2}{0}{q^{-n}, az }{-}{q,\frac{q^n}{z^2}}.
\end{eqnarray}

Replacing $q$ by $q^{-1}$ in the Askey--Wilson polynomials does not essentially change the polynomials. However, for the above given subfamilies of polynomials replacing $q$ by $q^{-1}$ leads to different kind of orthogonal polynomials. We also list the corresponding $q^{-1}$-subfamilies of polynomials. 

The continuous dual $q^{-1}$-Hahn polynomials can 
be obtained from the Askey--Wilson polynomials as follows
\begin{eqnarray}
&&\hspace{-3.5cm}p_n(x;a,b,c|q^{-1})=q^{-3\binom{n}{2}}
\left(-abc\right)^n
\lim_{d\to0}d^n\,p_n(x;\tfrac{1}{a},\tfrac{1}{b},
\tfrac{1}{c},\tfrac{1}{d}|q).
\label{limtcdqiH}
\end{eqnarray}
Then the continuous dual $q^{-1}$-Hahn polynomials are given by
\begin{eqnarray}
&&\hspace{-2.0cm}
\label{cdqiH:1} p_n(x;a,b,c|q^{-1})=
q^{-2\binom{n}{2}}
(abc)^n \left(
\frac{1}{ab},\frac{1}{ac}
;q\right)_n\qhyp{3}{2}{q^{-n},\frac{z^{\pm}}{a}}
{
\frac{1}{ab},\frac{1}{ac}
}{q,\frac{q^n}{bc}}.
\end{eqnarray}

\noindent The $q^{-1}$-Al-Salam--Chihara polynomials 
are given by 
\begin{eqnarray}
\label{qiASC:1}&&\hspace{-2.9cm}Q_n(x;a,b|q^{-1})=
q^{-\binom{n}{2}}
(-b)^n 
\left(\frac{1}{ab}
;q\right)_n\qhyp31{q^{-n}, 
\frac{z^{\pm}}{a}
}{
\frac{1}{ab}
}
{q,\frac{q^na}{b}}.
\end{eqnarray}
We will mostly be interested in the these polynomials in case $x=\frac12(z+z^{-1})$ with $z=q^{-m}a$, $m \in \N_0$. In this case we denote the polynomials by $Q_n[q^{-m}a;a,b|q^{-1}]$, and we have
\begin{equation} \label{q^{-1}ASC 2phi1}
Q_n[q^{-m}a;a,b|q^{-1}] = q^{-\binom{n}{2}}q^{-\binom{m}{2}}
\left(-\frac{a}{qb}\right)^m
(-b)^{n}
\frac{(\frac{1}{ab};q)_n(\frac{qb}{a};q)_m}{(\frac{1}{ab};q)_m}\qhyp21{q^{-m},\frac{q^{m}}{a^2}}{\frac{qb}{a}}{q,q^{n+1}}.
\end{equation}
Note that the $_2\phi_1$-function is a polynomial in $q^n$ of degree $m$, which can be identified as a little $q$-Jacobi polynomial. The $q^{-1}$-Al-Salam--Chihara polynomials have the following generating function
\begin{equation}
\sum_{n=0}^\infty
\frac{q^{\binom{n}{2}}t^n}{(q;q)_n}Q_n(x;a,b|q^{-1})=
\frac{(-tz^\pm;q)_\infty}
{(-ta,-tb;q)_\infty}.
\label{qiASCgf}
\end{equation}
If one chooses $z=q^{-m}a$, then one has
\begin{equation}
\sum_{n=0}^\infty\frac{q^{\binom{n}{2}}t^n}{(q;q)_n}Q_n[q^{-m}a;a,b|q^{-1}]=
q^{-\binom{m}{2}}
\frac{(-\frac{t}{a};q)_\infty}{(-bt;q)_\infty}
\frac{(-\frac{q}{at};q)_m}{(-\frac{t}{a};q)_m}\left(\frac{at}{q}\right)^m.
\label{argf}
\end{equation}
This generating function will be useful below.
The continuous big $q^{-1}$-Hermite polynomials are given by 
\begin{eqnarray}
\label{cbqiH:1} &&\hspace{-7.7cm}H_n(x;a|q^{-1}) =a^{-n}\qhyp{3}{0}{q^{-n},\frac{z^{\pm}}
{a}}{-}{q,q^na^2}.
\end{eqnarray}
Similar as for the $q^{-1}$-Al-Salam--Chihara polynomials we use the notation $H_n[aq^{-m};a|q^{-1}]$ for $H_n(x;a|q^{-1})$ with $x=\tfrac12(aq^{-m}+a^{-1}q^{m})$.

\subsection{Connection coefficients for Askey--Wilson polynomials}

The Askey--Wilson polynomials have the following connection relation due to Ismail and Zhang \cite[Theorem 3.2]{IsmailZhang2005} (see also \cite[Theorem 16.4.2]{Ismail}). 
\begin{thm}Let $n\in\N_0$, $q\in\CCddag$, $a,b,c,d,e,f,g,h\in\CCast$. Then
\begin{eqnarray}
&&\hspace{-0.8cm}p_n(x;a,b,c,d|q)=(q,ad,bd,cd;q)_n\sum_{k=0}^n
\frac{p_k(x;e,f,g,h|q)(q^kd)^{k-n}(q^{n-1}abcd;q)_k}
{(q;q)_{n-k}(q,ad,bd,cd,q^{k-1}efgh;q)_k}\nonumber\\
&&\hspace{-0.8cm}\times
\sum_{l=0}^{n-k}\left(\frac{qd}{h}\right)^l\frac{(q^{k-n},q^{n+k-1}abcd,q^keh,q^kfh,q^kgh;q)_l}{(q,q^{2k}efgh,q^kad,q^kbd,q^kcd;q)_l}
\!\qhyp43{q^{-(n-k-l)},q^{n+k+l-1}abcd,q^{k+l}dh,\frac{d}{h}}{q^{k+l}ad,q^{k+l}bd,q^{k+l}cd}{q,q}\!.
\label{conAW}
\end{eqnarray}
\label{conAWthm}
\end{thm}
\begin{proof}
See proof of \cite[Theorem 3.2]{IsmailZhang2005} and \cite[Theorem 16.4.2]{Ismail}.
\end{proof}
This connection coefficient simplifies to a ${}_5\phi_4$ in the case $d=h$, and this result was first given in the Askey--Wilson memoir \cite[\S6]{AskeyWilson85}. In the one free parameter case $(f,g,h)=(b,c,d)$ one obtains a ${}_3\phi_2$ which can be summed using $q$-Pfaff--Saalsch\"utz \cite[\href{http://dlmf.nist.gov/17.7.E5}{(17.7.5)}]{NIST:DLMF}. The result is as follows
\begin{eqnarray}
\hspace{-0.7cm}p_n(x;a,b,c,d|q)=
\frac{(q,ab,ac,bc;q)_n}
{(abce;q)_n}
\sum_{k=0}^n
p_k(x;a,b,c,e|q)
\frac{e^{n-k}(\frac{d}{e};q)_{n-k}(\frac{abce}{q},q^{n-1}abcd;q)_k(abce;q)_{2k}}{(q;q)_{n-k}(q,ab,ac,bc,q^nabce;q)_k(\frac{abce}{q};q)_{2k}}.
\label{AW1free}
\end{eqnarray}

\section{Nonstandard generating function for Askey--Wilson polynomials}
\label{NonstanAWgf}

In \cite{Groenevelt2004}, Groenevelt derived the following nonstandard generating function for Askey--Wilson polynomials, which can be considered as a generalization of the connection relation for Askey--Wilson polynomials \eqref{AW1free} in case $(e,f,g)=(a,b,c)$, see \cite[Remark 2.2]{Groenevelt2004}.

\begin{thm}
\label{nonstanthm}
Let $q\in\CCdag$, $x=\frac12(z+z^{-1})\in\C$, $a,b,c,d,t,r\in\CCast$, $|t|<\min\{|z|,|z|^{-1}\}$. Then
\begin{eqnarray}
&&\hspace{-1.7cm}\sum_{n=0}^\infty 
\frac{t^n\,(abcd;q)_{2n}(\frac{abcd}{q},\frac{r}{t},\frac{abc}{r};q)_n\,p_n(x;a,b,c,d|q)}{(\frac{abcd}{q};q)_{2n}(q,ab,ac,bc,\frac{abcdt}{r},dr;q)_n}\nonumber\\
&&\hspace{1cm}=\frac{(abcd,dt,\frac{abctz}{r},rz;q)_\infty}
{(\frac{abcdt}{r},dr,abcz,tz;q)_\infty}
\Whyp87{\frac{abcz}{q}}{az,bz,cz,\frac{r}{t},\frac{abc}{r}}{q,\frac{t}{z}}.
\label{nonstan}
\end{eqnarray}
\label{nsAWgf}
\end{thm}
\begin{proof}
See proof of \cite[Theorem 2.1]{Groenevelt2004}.
\end{proof}

\noindent There are some interesting limits of this nonstandard generating function. 
First we consider the limits which give nonstandard generating functions over Askey--Wilson polynomials. First there is the limit as $t\to 0$.

\begin{cor}
Let $q\in\CCdag$, $x=\frac12(z+z^{-1})\in\CCast$, $a,b,c,d,t\in\CCast$, $|t|<|c|$. Then 
\begin{eqnarray}
&&\sum_{n=0}^\infty
\frac{q^{\binom{n}{2}}t^n(abcd;q)_{2n}(\frac{abcd}{q},-\frac{abc}{t};q)_n}
{(\frac{abcd}{q};q)_{2n}(q,ab,ac,bc,-dt;q)_n}p_n(x;a,b,c,d|q)
=\frac{(abcd,-\frac{t}{c};q)_\infty}{(ab,-dt;q)_\infty}\qhyp32{-\frac{abc}{t},cz^\pm}{ac,bc}{q,-\frac{t}{c}}.
\end{eqnarray}
\label{cor42}
\end{cor}
\begin{proof}
Start with Theorem \ref{nsAWgf} and take the limit as $t\to 0$. Then replace $r\mapsto -r$. Finally taking advantage of the nonterminating transformation \eqref{W761tran} and replacing $r\mapsto t$ completes the proof. 
\end{proof}

\noindent Another interesting limit to Askey--Wilson polynomials is the limit as $t\to 0$ in the above nonstandard generating function. This produces the following result.

\begin{cor}
Let $q\in\CCdag$, $x=\frac12(z+z^{-1})\in\CCast$, $a,b,c,d\in\CCast$, $|c|<\min\{|z|,|z|^{-1}\}$. Then 
\begin{eqnarray}
&&\sum_{n=0}^\infty
\frac{q^{2\binom{n}{2}}(abc)^n(abcd;q)_{2n}(\frac{abcd}{q};q)_n}
{(\frac{abcd}{q};q)_{2n}(q,ab,ac,bc;q)_n}p_n(x;a,b,c,d|q)
=\frac{(abcd,\frac{c}{z};q)_\infty}{(ac,bc;q)_\infty}\qhyp21{az,bz}{ab}{q,\frac{c}{z}}.
\end{eqnarray}
\end{cor}
\begin{proof}
Start with Corollary \ref{cor42} and take the limit as $t\to 0$. This produces a ${}_2\phi_2$ on the right-hand side which can be converted to a ${}_2\phi_1$ using the nonterminating transformation  \eqref{rel2122}.
\end{proof}

\noindent There are also some interesting limits of the nonstandard generating function for Askey--Wilson polynomials which generate nonstandard generating functions for continuous dual $q$-Hahn polynomials. For instance, the $a\to 0$ limit gives the following result.

\begin{cor}
Let $q\in\CCdag$, $x=\frac12(z+z^{-1})\in\CCast$, $t,r,a,b,c\in\CCast$, $|t|<\min\{|z|,|z|^{-1}\}$. Then 
\begin{eqnarray}
\sum_{n=0}^\infty\frac{t^n\,(\frac{r}{t};q)_n\,p_n(x;a,b,c|q)}
{(q,bc,ar;q)_n}=\frac{(at,rz;q)_\infty}{(ar,tz;q)_\infty}\qhyp32{bz,cz,\frac{r}{t}}{bc,rz}{q,\frac{t}{z}}.
\end{eqnarray}
\label{cor44}
\end{cor}
\begin{proof}
Start with Theorem \ref{nsAWgf} and take the limit as $a\to 0$ followed by replacing $d\mapsto a$.
\end{proof}

\noindent There is also the limit as $t\to 0$ of the above result.
\begin{cor}
Let $q\in\CCdag$, $x=\frac12(z+z^{-1})\in\CCast$, $t,a,b,c\in\CCast$, $|t|<|b|$. Then 
\begin{eqnarray}
\sum_{n=0}^\infty\frac{q^{\binom{n}{2}}t^n\,p_n(x;a,b,c|q)}
{(q,bc,-at;q)_n}=\frac{(-tz;q)_\infty}{(-at;q)_\infty}\qhyp21{bz^\pm}{bc}{q,-\frac{t}{b}}.
\end{eqnarray}
\end{cor}
\begin{proof}
Start with Corollary \ref{cor44} and then take the limit as $t\to 0$ followed by replacing $r\mapsto t$ completes the proof.
\end{proof}

\noindent 
One also has a nice $d\to 0$ limit for the nonstandard Askey--Wilson generating function.
\begin{cor}
Let $q\in\CCdag$, $x=\frac12(z+z^{-1})\in\CCast$, $t,r,a,b,c\in\CCast$, $|t|<\min\{|z|,|z|^{-1}\}$. Then 
\begin{eqnarray}
&&\sum_{n=0}^\infty
\frac{t^n(\frac{r}{t},\frac{abc}{r};q)_n\,p_n(x;a,b,c|q)}{(q,ab,ac,bc;q)_n}=\frac{(\frac{abctz}{r},rz;q)_\infty}{(abcz,tz;q)_\infty}\Whyp87{\frac{abcz}{q}}{az,bz,cz,\frac{r}{t},\frac{abc}{r}}{q,\frac{t}{z}}.
\end{eqnarray}
\label{cor46}
\end{cor}
\begin{proof}
Start with Theorem \ref{nsAWgf} and take the limit as $d\to 0$.
\end{proof}

\noindent The $d\to 0$ limit followed by the $t\to 0$ limit has an interesting limit.
\begin{cor}
Let $q\in\CCdag$, $x=\frac12(z+z^{-1})\in\CCast$, $t,a,b,c\in\CCast$, $|t|<\min\{|z|,|z|^{-1}\}$. Then 
\begin{eqnarray}
&&\sum_{n=0}^\infty\frac{q^{\binom{n}{2}}t^n(-\frac{abc}{t};q)_n\,p_n(x;a,b,c|q)}
{(q,ab,ac,bc;q)_n}
=\frac{(-\frac{t}{c};q)_\infty}{(ab;q)_\infty}\qhyp32{-\frac{abc}{t},cz^\pm}{ac,bc}{q,-\frac{t}{z}}.
\end{eqnarray}
\end{cor}
\begin{proof}
Start with Corollary \ref{cor46} and then take the $t\to 0$ limit followed by replacing $r\mapsto t$.
\end{proof}

\section{Poisson kernel for $q^{-1}$-Al-Salam--Chihara polynomials}
\label{Poissonker}

One can derive a bilinear sum of a product of two $q^{-1}$-Al-Salam--Chihara polynomials which is a Poisson kernel.
\begin{thm} \label{thm:PoissonKernel}
Let $m,m'\in\N_0$, $q\in\CCdag$,  $a,b,c,d\in\CCast$, $t\in\CC$ such that $|bdt|<1$. Then
\begin{eqnarray}
&&\hspace{-0.7cm}P_{m,m'}(t;a,b,c,d):=\sum_{n=0}^\infty
\frac{q^{2\binom{n}{2}}t^n}{(q,\frac{1}{cd};q)_n}
Q_n[q^{-m}a;a,b|q^{-1}]
Q_n[q^{-m'}c;c,d|q^{-1}]\nonumber\\
&&\hspace{-0.2cm}=q^{-\binom{m}{2}}q^{-\binom{m'}{2}}\left(-\frac{c}{qd}\right)^{m'}
\left(-\frac{adt}{q}\right)^m
\frac{(\frac{dt}{a};q)_\infty}{(bdt;q)_\infty}
\frac{(\frac{q}{adt};q)_m}{(\frac{dt}{a};q)_m}
\frac{(\frac{qd}{c};q)_{m'}}{(\frac{1}{cd};q)_{m'}}
\qhyp43{q^{-m'},\frac{q^{m'}}{c^2},adt,bdt}{q^{-m}adt,q^m\frac{dt}{a},\frac{qd}{c}}{q,q}.
\end{eqnarray}
\label{bilinASCQi}
\end{thm}
\begin{proof}
Start with the bilinear sum
for $P_{m,m'}(t;a,b,c,d):=P_{m,m'}(t)$. From \eqref{q^{-1}ASC 2phi1} it follows that 
\[
Q_n[q^{-m}a;a,b|q^{-1}] = \mathcal O(b^{n} q^{-\binom{n}{2}})
\]
for $n \to \infty$, so that the sum converges absolutely for $|bdt|<1$. Then replace one of the $q^{-1}$-Al-Salam--Chihara polynomials with its ${}_2\phi_1$ representation using
\eqref{q^{-1}ASC 2phi1}. Replace the ${}_2\phi_1$ with it's infinite series representation and reverse the order of the sums. This produces
\begin{equation}
\hspace{-0.2cm}P_{m,m'}(t)=q^{-\binom{m'}{2}}
\left(-\frac{c}{qd}\right)^{m'}
\frac{(\frac{qd}{c};q)_{m'}}{(\frac{1}{cd};q)_{m'}}
\sum_{k=0}^\infty q^k
\frac{(q^{-m'},\frac{q^{m'}}{c^2};q)_k}{(q,\frac{qd}{c};q)_k}
\sum_{n=0}^\infty
\frac{q^{\binom{n}{2}}(-q^kdt)^n}{(q;q)_n}Q_n[q^{-m}a;a,b|q^{-1}],
\label{interm}
\end{equation}
in which the sum over $n$ is a rescaled form of the generating function 
\eqref{argf}. Inserting this result in \eqref{interm} completes the proof.
\end{proof}

\begin{rem}
Note that the $_4\phi_3$ in Theorem \ref{thm:PoissonKernel} becomes balanced if $ab=cd$. In this case there is an obvious symmetry: $P_{m,m'}(t;a,b,c,d) = P_{m',m}(t;c,d,a,b)$.
\label{rem32}
\end{rem}

\noindent If you scale the variable $t$ and then take the limit as $b,d\to 0$, one can obtain a corresponding bilinear sum for continuous big $q^{-1}$-Hermite polynomials.
\begin{cor}
Let $m,m'\in\N_0$, $q\in\CCdag$,  $a,b\in\CCast$, $t\in\CC$.
Then
\begin{eqnarray}
&&\hspace{-0.7cm}
\sum_{n=0}^\infty
\frac{q^{\binom{n}{2}}t^n}{(q;q)_n}
H_n[q^{-m}a;a|q^{-1}]
H_n[q^{-m'}b;b|q^{-1}]\nonumber\\
&&\hspace{-0.2cm}=q^{-\binom{m}{2}}q^{-2\binom{m'}{2}}\left(\frac{b^2}{q}\right)^{m'}
\left(\frac{at}{qb}\right)^m
(-\tfrac{t}{ab};q)_\infty
\frac{(-\frac{qb}{at};q)_m}{(-\frac{t}{ab};q)_m}
\qhyp32{q^{-m'},\frac{q^{m'}}{b^2},-\frac{at}{b}}{-q^{-m}\frac{at}{b},-q^m\frac{t}{ab}}{q,q}.
\end{eqnarray}
\end{cor}
\begin{proof}
Evaluate $P_{m,m'}(\tfrac{t}{d};a,b,c,d)$
using Theorem \ref{bilinASCQi}. Then take the limit as $b,d\to 0$. Finally replacing $c\mapsto b$ completes the proof.
\end{proof}

\section{A bilinear sum for Askey--Wilson polynomials}
\label{binlinsec}
The expansion of the Askey--Wilson polynomial in terms of Al-Salam--Chihara and $q^{-1}$-Al-Salam--Chihara polynomials is given as follows.
\begin{thm}{Groenevelt (2021) \cite{Groenevelt2021}.} \label{thm:Groenevelt}
Let $n\in\N_0$, $x=\frac12(z+z^{-1})\in\CCast$, $q\in\CCdag$, $a,b,c,d\in\CCast$. Then 
\begin{eqnarray}
&&\hspace{-1cm}p_n(x;a,b,c,d|q)=
q^{\binom{n}{2}}(-d)^n
\frac{(cz^\pm;q)_\infty(ab,ac,bc;q)_n}
{(ac,bc;q)_\infty}\nonumber\\
&&\hspace{1cm}\times
\sum_{k=0}^\infty \frac{q^{\binom{k}{2}}(-1)^k}{(q,ab;q)_k}\left(\frac{qabc}{d}\right)^{\frac{k}{2}}
Q_k\left[\frac{q^{\frac12-n}}{\sqrt{abcd}};\frac{\sqrt{q}}{\sqrt{abcd}},\frac{\sqrt{cd}}{\sqrt{qab}}\biggl|q^{-1}\right]Q_k(x;a,b|q).
\end{eqnarray}
\label{GQQiexp}
\end{thm}
\begin{proof}
See proof of \cite[Lemma 4.6]{Groenevelt2021}.
\end{proof}

For $0<q<1$, $ab>1$ and $qb<a$ the $q^{-1}$-Al-Salam--Chihara polynomials satisfy the following orthogonality relations
\begin{equation} \label{ASCqi orth}
    \sum_{m=0}^\infty Q_k[q^{-m}a;a,b|q^{-1}] Q_{k'}[q^{-m}a;a,b|q^{-1}] W_m(a,b;q) = \delta_{k,k'}\, N_n(a,b;q),
\end{equation}
with 
\begin{equation} \label{ASCqiO}
\begin{split}
W_m(a,b;q)&= \frac{ (\frac{qb}{a};q)_\infty }{(qa^{-2};q)_\infty} q^{m^2}
\frac{1-q^{2m}a^{-2}}{1-a^{-2}}\frac{(a^{-2},\frac1{ab};q)_m }{ (q,\frac{qb}{a};q)_m } \left( \frac{b}{a} \right)^m , \\
N_k(a,b;q)&=q^{-2\binom{k}{2}}\left(\frac{ab}{q}\right)^{k} \left(q, \frac1{ab};q\right)_k.
\end{split}
\end{equation}
These orthogonality relations are essentially the dual orthogonality relations of the little $q$-Jacobi polynomials.
Using the above expansion and applying the orthogonality relation, we can obtain an expansion for the Al-Salam--Chihara polynomials as a bilinear sum over a product of $q^{-1}$-Al-Salam--Chihara polynomials and Askey--Wilson polynomials. 

\medskip
\begin{lem}
\label{lemASCAWQi}
Let $n\in\N_0$, $0<q<1$, $a,b,c,d\in\CCast$ such that $0<ab,cd<1$, and $x=\frac12(z+z^{-1})\in\CC$, then
\begin{eqnarray}
&&\sum_{k=0}^\infty q^{\binom{k}{2}}(-c)^k\frac{(\frac{abcd}{q},\pm\sqrt{qabcd};q)_k}{(q,\pm\sqrt{\frac{abcd}{q}},ac,bc,cd;q)_k}
Q_n\left[\frac{q^{1/2-k}}{\sqrt{abcd}};\sqrt{\frac{q}{abcd}},\sqrt{\frac{cd}{qab}}\bigg|q^{-1}\right]
p_k(x;a,b,c,d|q)\nonumber\\
&&\hspace{6cm}=q^{-\binom{n}{2}}\frac{(abcd,cz^\pm;q)_\infty}
{(ac,bc,cd;q)_\infty}\left(-\sqrt{\frac{c}{qabd}}\right)^n
Q_n(x;a,b|q).
\label{bilinASC}
\end{eqnarray}
\end{lem}
\begin{proof}
Start with Theorem \ref{GQQiexp}
and replace 
\begin{equation}
\left(\sqrt{\frac{q}{abcd}},
\sqrt{\frac{cd}{qab}},
\sqrt{\frac{bd}{qac}},
\sqrt{\frac{bc}{qad}}\right)\mapsto(a,b,c,d).
\label{map}
\end{equation}
Then multiply both sides of the equation by the weight function in \eqref{ASCqiO} multiplied by $Q_{k'}[q^{-n}a;a,b|q^{-1}]$, 
$k'\in\N_0$, and sum over all $n\in\N_0$ utilizing the orthogonality relation. This selects out the $k=k'$ term in the infinite sum over $k$. Then replacing $k\leftrightarrow n$, and performing the inverse to the map \eqref{map} throughout, completes the proof. 
\end{proof}

\medskip
There are some interesting special cases which can be derived from this expansion. First for continuous dual $q$-Hahn polynomials
\begin{eqnarray}
&&p_m(x;a,b,c|q)
=q^{\binom{m}{2}}(-c)^m(ab;q)_m
\sum_{n=0}^m
\left(\frac{q}{c}\right)^n
\frac{(q^{-m};q)_n}{(q,ab;q)_n}Q_n(x;a,b|q)\\
&&\hspace{2.6cm}=q^{\binom{m}{2}}(-c)^m\frac{(bz^\pm;q)_\infty(ab;q)_m}{(ab;q)_\infty}
\sum_{n=0}^\infty
\frac{\left(qb/c\right)^{\frac{n}{2}}}
{(q;q)_n}
H_n\left[\sqrt{\frac{q}{bc}};\frac{q^{\frac12-m}}{\sqrt{bc}}\biggl|q\right]
H_n(x;a|q),
\end{eqnarray}
and for Al-Salam--Chihara polynomials
\begin{eqnarray}
&&\hspace{-1cm}Q_m(x;a,b|q)=q^{\binom{m}{2}}(-b)^m(az^\pm;q)_\infty
\sum_{n=0}^\infty
\frac{\left(qa/b\right)^{\frac{n}{2}}}{(q;q)_n}
H_n\left[\sqrt{\frac{q}{ab}};\frac{q^{\frac12-m}}{\sqrt{ab}}\biggl|q\right]
H_n(x|q),
\end{eqnarray}
and for continuous big $q$-Hermite polynomials
\begin{eqnarray}
&&H_m(x;a|q)
=q^{\binom{m}{2}}(-a)^m
\sum_{n=0}^m
\left(\frac{q}{a}\right)^n
\frac{(q^{-m};q)_n}{(q;q)_n}H_n(x|q),
\end{eqnarray}
where $H_n(x|q)$ is the continuous $q$-Hermite polynomial.

\section{A general expansion for Askey--Wilson polynomials}
\label{genexpan}

We are able to derive a bilinear sum over Askey--Wilson polynomials 
which generalizes the expansion given in Theorem \ref{GQQiexp}.
\begin{thm} \label{Thm:sum AWpol}
Let $n\in\N_0$, $q\in\CCdag$, $z,a,b,c,d,e,f\in\CCast$, such that $\max\{|cz|, |{c}/{z}|\}<1$. Then 
\begin{eqnarray}
&&\hspace{-0.4cm}p_{n}(x;a,b,c,d|q)=
\frac{(ae,be,ef,abcf,cz^\pm;q)_\infty}{(ac,bc,cf,abef,ez^\pm;q)_\infty}
(ac,bc,cd;q)_{n}\nonumber\\
&&\hspace{.8cm}\times
\sum_{k=0}^\infty
c^{k-n} 
\frac{(abef;q)_{2k}(\frac{e}{c},\frac{abef}{q};q)_k}
{(\frac{abef}{q};q)_{2k}(q,ae,be,ef,abcf;q)_k}
p_k(x;a,b,e,f|q)
\qhyp43{q^{-n},q^{n-1}abcd,cf,\frac{qc}{e}}
{cd,q^kabcf,q^{1-k}\frac{c}{e}}{q,q}
\label{thmGAWeq}\\
&&=
\frac{(ae,be,ef,abcf,cz^\pm;q)_\infty}{(ac,bc,cf,abef,ez^\pm;q)_\infty}
\frac{(ab,ac,bc,\frac{d}{f};q)_{n}}{(abcf;q)_n}\nonumber\\
&&\hspace{.8cm}\times
\sum_{k=0}^\infty
f^{k-n} 
\frac{(abef;q)_{2k}(\frac{abef}{q};q)_k}
{(\frac{abef}{q};q)_{2k}(q,ab,ae,be;q)_k}
p_k(x;a,b,e,f|q)
\qhyp43{q^{-k},q^{k-1}abef,cf,\frac{qf}{d}}
{ef,q^nabcf,q^{1-n}\frac{f}{d}}{q,q}.
\label{thmGAWeq2}
\end{eqnarray}
\label{thmGAW}
\end{thm}
\begin{proof}
Start with Lemma \ref{lemASCAWQi}. Replace $c\mapsto e$ and $d\mapsto f$ and
divide both sides by the prefactor of the Al-Salam--Chihara polynomials on the right-hand side. Multiply both sides by 
\[
q^{\binom{m}{2}}(-d)^m
\frac{(cz^\pm;q)_\infty(ab,ac,bc;q)_m}
{(ac,bc;q)_\infty} \frac{q^{\binom{n}{2}}(-1)^n}{(q,ab;q)_n}\left(\frac{qabc}{d}\right)^{\frac{n}{2}}
Q_n\left[\frac{q^{\frac12-m}}{\sqrt{abcd}};\frac{\sqrt{q}}{\sqrt{abcd}},\frac{\sqrt{cd}}{\sqrt{qab}}\biggl|q^{-1}\right].
\]
and then sum both sides of the resulting expression over $n\ge 0$. Applying Theorem \ref{GQQiexp} on the right-hand side gives the Askey--Wilson polynomial $p_m(x;a,b,c,d|q)$. The left-hand side is a double sum over $n$ and $k$. Interchanging the $n$ and $k$ sums and identifying the sum over $n$ using Theorem \ref{bilinASCQi} with 
\[
\{m',m,a,b,c,d,t\}\mapsto\left\{m,k,\sqrt{\frac{q}{abef}},\sqrt{\frac{ef}{qab}},\sqrt{\frac{q}{abcd}},\sqrt{\frac{cd}{qab}},qab\sqrt{\frac{cf}{de}}\right\}
\]
and replacing $m\mapsto n$ in the resulting expression produces \eqref{thmGAWeq}.
Alternatively, by utilizing the apparent symmetry in Remark \ref{rem32}, 
interchanging the $n$ and $k$ sums and identifying the sum over $n$ using Theorem \ref{bilinASCQi} with 
\[
\{m',m,c,d,a,b,t\}\mapsto\left\{k,m,\sqrt{\frac{q}{abef}},\sqrt{\frac{ef}{qab}},\sqrt{\frac{q}{abcd}},\sqrt{\frac{cd}{qab}},qab\sqrt{\frac{cf}{de}}\right\}
\]
and replacing $m\mapsto n$ in the resulting expression produces 
\eqref{thmGAWeq2}. In order to determine the range of parameters over which the infinite series \eqref{thmGAWeq}, \eqref{thmGAWeq2} respectively are convergent, write the right-hand sides respectively as
\begin{eqnarray}
&&\hspace{-0.5cm}p_{m}(x;a,b,c,d|q)=
\frac{(abcf,ae,be,ef,cz^\pm;q)_\infty}{(abef,ac,bc,fc,ez^\pm;q)_\infty}
(ac,bc,cd;q)_{m}\sum_{k=0}^\infty G_{k,m}(z;a,b,c,d,e,f|q),
\end{eqnarray} 
\begin{eqnarray}
p_m(x;a,b,c,d|q)=\frac{(ae,be,ef,abcf,cz^\pm;q)_\infty}{(ac,bc,cf,abef,ez^\pm;q)_\infty}
\frac{(ab,ac,bc,\frac{d}{f};q)_{m}}{(abcf;q)_m}\sum_{k=0}^\infty H_{k,m}(z;a,b,c,d,e,f|q).
\end{eqnarray}
It is then straightforward to show using Ismail's asymptotic expansion for large degree of Askey--Wilson polynomials \cite[\href{http://dlmf.nist.gov/18.29.E1}{(18.29.1)}]{NIST:DLMF} that as $k\to\infty$ one has respectively 
\begin{eqnarray}
&&\hspace{-1.6cm}G_{k,m}(z;a,b,c,d,e,f|q)\sim\frac{(\frac{e}{c},abef;q)_\infty\,c^{-m}}{(q,ae,be,ef,abcf;q)_\infty(\frac{abcd}{q};q)_{2m}}\nonumber\\
&&\hspace{3.0cm}\times\left((cz)^k
\frac{(\frac{a}{z},\frac{b}{z},\frac{e}{z},\frac{f}{z};q)_\infty}{(z^{-2};q)_\infty}+\left(\frac{c}{z}\right)^k\frac{(az,bz,ez,fz;q)_\infty}{(z^2;q)_\infty}\right),\\
&&\hspace{-1.6cm}H_{k,m}(z;a,b,c,d,e,f|q)\sim\frac{(\frac{e}{c},abef;q)_\infty(cd,abcf;q)_m\,c^{-m}}{(q,ae,be,ef,abcf;q)_\infty(ab,\frac{d}{f};q)_m}\nonumber\\
&&\hspace{3.0cm}\times\left((cz)^k
\frac{(\frac{a}{z},\frac{b}{z},\frac{e}{z},\frac{f}{z};q)_\infty}{(z^{-2};q)_\infty}+\left(\frac{c}{z}\right)^k\frac{(az,bz,ez,fz;q)_\infty}{(z^2;q)_\infty}\right),
\end{eqnarray}
which provides the restrictions on the parameters $c$ and $z$ in order for the infinite sum to converge.
Furthermore, using the identity \eqref{inffinrel} 
and
\[
\begin{split}
&\frac{(abcf;q)_\infty (\frac{e}{c};q)_k (cd;q)_n }{ (abcf;q)_k} \qhyp43{q^{-n},q^{n-1}abcd,cf,\frac{qc}{e}}
{cd,q^kabcf,q^{1-k}\frac{c}{e}}{q,q} \\
&= \sum_{j=0}^n \frac{(q^{-n}, q^{n-1}abcd, cf;q)_j (cdq^j;q)_{n-j} (abcfq^{k+j};q)_\infty q^j}{(q;q)_j} \times
\begin{cases}
\left( - \tfrac{ eq^k }{c} \right)^j q^{-\binom{j+1}{2}} (\tfrac{e}{c};q)_{k-j} (\tfrac{qc}{e};q)_j, &j \leq k,\\
\left( - \tfrac{ eq^k }{c} \right)^k q^{-\binom{k+1}{2}} ( \tfrac{c}{e} q^{j-k+1};q)_k, & j > k,
\end{cases}
\end{split}
\]
it follows that the summand is analytic in $z,a,b,c,d,e,f$ on $\CCast$ and in $q$ on $\CCdag$. So by uniform convergence on bounded closed sets the sum is also analytic in all these parameters provided $\max\{|cz|,|c/z|\}<1$. Then by analytic continuation we can get rid of the conditions on the parameters needed for Lemma \ref{lemASCAWQi}.
This completes the proof. 
\end{proof}


\begin{rem}
The above expansion has an evident symmetry in $a,b$.
\end{rem}

\begin{rem}
It is interesting to compare the above infinite expansion for the Askey--Wilson polynomials with the connection coefficients for Askey--Wilson polynomials which are obtained by setting $h=d$ and then replacing $a\leftrightarrow d$ followed by setting $g=b$ in \eqref{conAW}, namely 
\begin{eqnarray}
&&\hspace{-0.2cm}p_n(x;a,b,c,d|q)=(q,ab,ac,ad;q)_n
\sum_{k=0}^np_k(x;a,b,e,f|q)\nonumber\\
&&\hspace{1.7cm}\times\frac{\left(q^ka\right)^{k-n}(q^{n-1}abcd;q)_k}{(q;q)_{n-k}(q,ab,ac,ad,q^{k-1}abef;q)_k}
\qhyp43{q^{k-n},q^kae,q^kaf,q^{n+k-1}abcd}{q^kac,q^kad,q^{2k}abef}{q,q}.
\label{conAW2free}
\end{eqnarray}
Both expansions involve the Askey--Wilson polynomial $p_k(x;a,b,e,f|q)$ multiplied by terminating ${}_4\phi_3$'s, except \eqref{conAW2free} truncates because of the prefactor $(q;q)_{n-k}$ term in the denominator.
\end{rem}
\noindent One interesting limit for the above expansion is the limit $c\to e$. In this case, the expansion becomes the connection coefficients for the Askey--Wilson polynomials with one free parameter, which is well known \cite[(6.5)]{AskeyWilson85} to be a simple product \eqref{AW1free}. Examining this limit give the following interesting result.
\begin{cor}
Let $n,k\in\N_0$ such that $n\ge k$, $q\in\CCddag$,  $a,b,c\in\CC$. Then
\begin{eqnarray}
&&\lim_{x\to 1}
(x;q)_k\qhyp43{q^{-n},q^{n-1}ab,c,\frac{q}{x}}{q^kac,b,\frac{q^{1-k}}{x}}{q,q}=c^{n-k}
\frac{(q,a;q)_n(\frac{b}{c};q)_{n-k}(c,ac,q^{n-1}ab;q)_k}{(b,ac;q)_n(q;q)_{n-k}(a,q^nac;q)_k}
\end{eqnarray}
\end{cor}
\begin{proof}
Take the limit as $e\to c$ in Theorem \ref{thmGAW} and the contribution from $n\ge k$ vanishes. Then compare the limit with the connection coefficient with one free parameter \eqref{AW1free}. Then after straightforward substitutions, the result follows.
\end{proof}

\noindent If one starts with \eqref{thmGAWeq2} and sets $e\mapsto c$ followed by $f\mapsto e$, one obtains the following result. 
\begin{cor}Let $n\in\N_0$, $q\in\CCdag$, $z,a,b,c,d,e\in\CCast$, such that $\max\{|ez|,|e/z|\}<1$. Then one has the following connection relation for Askey--Wilson polynomials 
\begin{eqnarray}
&&\hspace{-1.2cm}p_n(x;a,b,c,d|q)=\frac{(ab,ac,bc,\frac{d}{e};q)_n}{(abce;q)_n}\nonumber\\
&&\times\sum_{k=0}^\infty
\frac{e^{k-n}(abce;q)_{2k}(\frac{abce}{q};q)_k}{(\frac{abce}{q};q)_{2k}(q,ab,ac,bc;q)_k}p_k(x;a,b,c,e|q)\qhyp32{q^{-k},q^{k-1}abce,\frac{qe}{d}}{q^nabce,\frac{q^{1-n}e}{d}}{q,q}.
\label{thm55eq}
\end{eqnarray}
\end{cor}

\noindent 
Then if one starts with \eqref{thm55eq} and  takes the limit as $d\to 0$ followed by $e\mapsto d$, one obtains the following expansion of the continuous dual $q$-Hahn polynomials in terms of the Askey--Wilson polynomials.

\begin{cor}Let $n\in\N_0$, $q\in\CCdag$, $a,b,c,d\in\CCast$. Then
\begin{eqnarray}
p_n(x;a,b,c|q)=\frac{(ab,ac,bc;q)_n\,d^n}{(abcd;q)_n}\sum_{k=0}^n
\frac{q^{-\binom{k}{2}}\left(-\frac{q^n}{d}\right)^k(abcd;q)_{2k}(q^{-n},\frac{abcd}{q};q)_k}{(\frac{abcd}{q};q)_{2k}(q,ab,ac,bc,q^nabcd;q)_k}p_k(x;a,b,c,d|q)
\end{eqnarray}
\end{cor}

\begin{proof}
Start with \eqref{thm55eq}, 
take the limit as $d\to 0$, replace $e\mapsto d$. In this limit the ${}_3\phi_2$ becomes a ${}_2\phi_1$ which can then be evaluated using the $q$-Gauss sum \cite[\href{http://dlmf.nist.gov/17.6.E1}{(17.6.1)}]{NIST:DLMF}
which after simplification completes the proof.
\end{proof}

\noindent Note that the limits $a\to 0$ and $b\to 0$ in \eqref{thm55eq} are equivalent and produce the following terminating connection relation for continuous dual $q$-Hahn polynomials.
\begin{cor}Let $n\in\N_0$, $q\in\CCdag$, $a,b,c,d\in\CCast$. Then
\begin{eqnarray}
&&\hspace{-4cm}p_n(x;a,b,c|q)=d^n(bc,\tfrac{a}{d};q)_n\sum_{k=0}^n\left(\frac{q}{a}\right)^k\frac{(q^{-n};q)_k}{(q,bc,\frac{q^{1-n}d}{a};q)_k}p_k(x;b,c,d|q).
\end{eqnarray}
\end{cor}
\begin{proof}
Take the limit as $a\to 0$ in \eqref{thm55eq} and then relabeling the parameters completes the proof.
\end{proof}

\medskip
\noindent There are some other interesting limits of Theorem \ref{thmGAW}. First consider some expansions of the Askey--Wilson polynomial whose expansion coefficients are the continuous dual $q$-Hahn polynomials.

\begin{cor}
Let $n\in\N_0$, $q\in\CCdag$, $a,b,c,d,e\in\CCast$. Then
\begin{eqnarray}
&&\hspace{-1.0cm}p_{n}(x;a,b,c,d|q)=
\frac{(ae,be,cz^\pm;q)_\infty}{(ac,bc,ez^\pm;q)_\infty}
(ac,bc,cd;q)_{n}
\nonumber\\
&&\hspace{3.0cm}\times
\sum_{k=0}^\infty
\frac{c^{k-n}(\frac{e}{c};q)_k}
{(q,ae,be;q)_k}
p_k(x;a,b,e|q)
\qhyp32{q^{-n},q^{n-1}abcd,\frac{qc}{e}}
{cd,q^{1-k}\frac{c}{e}}{q,q}\label{AWcdqH1}\\
&&\hspace{1.7cm}=
\frac{(abce,cz^\pm;q)_\infty}{(ac,bc,ec;q)_\infty}
(ac,bc,cd;q)_{n}\nonumber\\
&&\hspace{3.0cm}\times
\sum_{k=0}^\infty
\frac{c^{k-n} }
{(q,abce;q)_k}
p_k(x;a,b,e|q)
\qhyp32{q^{-n},q^{n-1}abcd,ce}
{cd,q^kabce}{q,q^{k+1}}\label{AWcdqH2}\\
&&\hspace{1.7cm}
=\frac{(abcd,cz^\pm;q)_\infty}{(ac,bc,cd;q)_\infty}\frac{(ab,bc,cd;q)_n}{(abcd;q)_{2n}}\sum_{k=0}^\infty \frac{c^k}{(q,q^{2n}abcd;q)_k}p_{n+k}(x;a,b,d|q)
\label{AWcdqH3}\\
&&\hspace{1.7cm}
=\frac{(\frac{qa}{d},\frac{qb}{d},cz^\pm;q)_\infty}{(ac,bc,\frac{q}{d}z^\pm;q)_\infty}q^{\binom{n}{2}}(-d)^n(ab,ac,bc;q)_n\sum_{k=0}^\infty \frac{c^k(q^nab,\frac{q^{1-n}}{cd};q)_k}{(q,ab,\frac{qa}{d},\frac{qb}{d};q)_k}p_k(x;a,b,\tfrac{q}{d}|q).
\label{AWcdqH4}
\end{eqnarray}
\label{tAWcdqH1}
\end{cor}
\begin{proof}
The first result is obtained by taking the limit as $f\to 0$ in Theorem \ref{thmGAW}. The second is obtained by taking the limit as $e\to 0$ in Theorem \ref{thmGAW} and then replacing $f\mapsto e$. The third result is obtained by setting $e=d$ in the second result,in which case the $_3\phi_2$ becomes a $_2\phi_1$ that can be evaluated using
the $q$-Chu--Vandermonde sum \cite[\href{http://dlmf.nist.gov/17.6.E2}{(17.6.2)}]{NIST:DLMF}. The resulting expression only survives for $k\ge n$. Then shifting the summation index gives \eqref{AWcdqH3}.
The fourth result is obtained by setting $e=q/d$ in the first result. One can then simplify the resulting expression using the $q$-Chu--Vandermonde sum \eqref{qChuVander} 
which provides \eqref{AWcdqH4}.
\end{proof}

\begin{rem}
If one considers expansions of the Askey--Wilson polynomial whose expansion coefficients are Al-Salam--Chihara polynomials, 
it is straightforward to show that if one takes the limit $e,f\to 0$ in Theorem \ref{thmGAW}, then one can see that it is a generalization of Theorem \ref{GQQiexp}.
\end{rem}

\noindent We also have an expansion of a continuous dual $q$-Hahn polynomial in terms of Askey--Wilson polynomials.

\begin{cor}
Let $n\in\N_0$, $q\in\CCdag$, $x=\frac12(z+z^{-1})\in\CCast$, $a,b,c,d,e\in\CCast$. Then
\begin{eqnarray}
&&\hspace{-1.0cm}p_{n}(x;a,b,c|q)=
\frac{(abce,ad,bd,ed,cz^\pm;q)_\infty}{(abde,ac,bc,ec,dz^\pm;q)_\infty}
(ac,bc;q)_{n}\nonumber\\
&&\hspace{.2cm}\times
\sum_{k=0}^\infty
\frac{c^{k-n} (abde;q)_{2k}(\frac{abde}{q},\frac{d}{c};q)_k}
{(\frac{abde}{q};q)_{2k}(q,ad,bd,ed,abce;q)_k}
p_k(x;a,b,d,e|q)
\qhyp32{q^{-n},ce,\frac{qc}{d}}
{q^kabce,q^{1-k}\frac{c}{d}}{q,q}.
\end{eqnarray}
\end{cor}
\begin{proof}
This result follows by starting with Theorem \ref{thmGAW} and taking the limit as $d\to 0$, followed by replacing $(e,f)\mapsto(d,e)$.
\end{proof}

\noindent Now
we consider expansions of continuous dual $q$-Hahn polynomials in terms of continuous dual $q$-Hahn polynomials.

\begin{cor}
Let $n\in\N_0$, $q\in\CCdag$, $x=\frac12(z+z^{-1})\in\CCast$, $a,b,c,d,e\in\CCast$. Then
\begin{eqnarray}
&&\hspace{-0.65cm}p_{n}(x;a,b,c|q)=
\frac{(ae,ed,cz^\pm;q)_\infty}{(ac,cd,ez^\pm;q)_\infty}
(ac,bc;q)_{n}
\sum_{k=0}^\infty
\frac{c^{k-n}(\frac{e}{c};q)_k}
{(q,ae,de;q)_k}
p_k(x;a,d,e|q)
\qhyp32{q^{-n},\frac{qc}{e},cd}
{bc,q^{1-k}\frac{c}{e}}{q,q}\label{cdqHcdqH1}\\
&&\hspace{1.7cm}=
\frac{(ad,bd,cz^\pm;q)_\infty}{(ac,bc,dz^\pm;q)_\infty}\frac{(ac,bc;q)_n}{(ad,bd;q)_{n}}\sum_{k=0}^\infty \frac{c^k(\frac{d}{c};q)_k}{(q,q^{n}ad,q^nbd;q)_k}p_{n+k}(x;a,b,d|q)
\label{cdqHcdqH2}\\
&&\hspace{1.7cm}=
\frac{(abcd,cz^\pm;q)_\infty}{(ac,bc,cd;q)_\infty}
(ac,bc;q)_{n}
\sum_{k=0}^\infty
\frac{c^{k-n}}
{(q,abcd;q)_k}
p_k(x;a,b,d|q)
\qhyp21{q^{-n},cd}
{q^kabcd}{q,q^{k+1}}.\label{cdqHcdqH3}
\end{eqnarray}
\label{tAWcdqH2}
\end{cor}
\begin{proof}
The first result is obtained by taking the limit as $a\to 0$ in Theorem \ref{thmGAW} followed by replacing 
$d\mapsto a$, and $a\leftrightarrow b$, $f\mapsto d$. The second result follows by taking the limit as $d\to 0$ in \eqref{AWcdqH1} followed by replacing $e\mapsto d$, which evaluates to a $_2\phi_1$ which can be evaluated using the $q$-Chu--Vandermonde sum \eqref{qChuVander}.
Notice that the sum only survives for $k\ge n$. After shifting the summation index one obtains \eqref{cdqHcdqH2}. The third result follows by setting $d\to 0$ in \eqref{AWcdqH2} and replacing $e\mapsto d$.
\end{proof}


\noindent If you start with \eqref{AWcdqH4}, replace $(a,b,c)\mapsto(\frac{1}{a},\frac{1}{b},\frac{1}{c})$ using \eqref{limtcdqiH} and take the $d\to 0$ limit, then you obtain the following expansion of a continuous dual $q^{-1}$-Hahn polynomial in terms of Al-Salam--Chihara polynomials.
\begin{cor}
Let $n\in\N_0$, $q\in\CCdag$, $x=\frac12(z+z^{-1})\in\CCast$, $a,b,c\in\CCast$. Then
\begin{eqnarray}
p_n(x;a,b,c|q^{-1})=q^{-2\binom{n}{2}}
\frac{(\frac{1}{c}z^\pm;q)_\infty}{(\frac{1}{ac},\frac{1}{bc};q)_\infty}
(abc)^n\left(\tfrac{1}{ab},\tfrac{1}{ac},\tfrac{1}{bc};q\right)_n
\sum_{k=0}^\infty
\left(\frac{1}{c}\right)^k
\frac{(\frac{q^n}{ab};q)_k}{(q,\frac{1}{ab};q)_k}Q_k(x;\tfrac{1}{a},\tfrac{1}{b}|q).
\end{eqnarray}
\end{cor}
\begin{proof}
First start with \eqref{AWcdqH4}, replace $(a,b,c)\mapsto(\frac{1}{a},\frac{1}{b},\frac{1}{c})$, multiply both sides of the resulting equation by $q^{-3\binom{n}{2}}(-abcd)^n$, use \eqref{limtcdqiH} and take the $d\to 0$ limit. Then using the terminating $q$-binomial theorem \eqref{termbinom} completes the proof.
\end{proof}

\noindent Now
we consider expansions of Al-Salam--Chihara polynomials in terms of Al-Salam--Chihara polynomials.

\begin{cor}
Let $n\in\N_0$, $q\in\CCdag$, $x=\frac12(z+z^{-1})\in\C$, $a,b,c,d\in\CCast$. Then
\begin{eqnarray}
&&Q_n(x;a,b|q)=
\frac{(cd,az^\pm;q)_\infty}{(ad,cz^\pm;q)_\infty}(ab;q)_n
\sum_{k=0}^\infty \frac{a^{k-n}(\frac{c}{a};q)_k}{(q,cd;q)_k}Q_k(x;c,d|q)
\qhyp32{q^{-n},\frac{qa}{c},ad}{q^{1-k}\frac{a}{c},ab}{q,q}\label{ASCASC1}\\
&&\hspace{2.15cm}=
\frac{(bc,az^\pm;q)_\infty}{(ab,cz^\pm;q)_\infty}\frac{(ab;q)_n}{(bc;q)_n}
\sum_{k=0}^\infty \frac{a^{k}(\frac{c}{a};q)_k}{(q,q^nbc;q)_k}Q_{n+k}(x;b,c|q)
\label{ASCASC2}\\
&&\hspace{2.15cm}=
\frac{(bz^\pm;q)_\infty}{(ab,bc;q)_\infty}(ab;q)_n
\sum_{k=0}^\infty \frac{b^{k-n}}{(q;q)_k}Q_k(x;a,c|q)
\qhyp21
{q^{-n},bc}{0}{q,q^{k+1}}.\label{ASCASC3}
\end{eqnarray}
\end{cor}
\begin{proof}
The first result is obtained by starting  with \eqref{cdqHcdqH1} and taking the limit as $a\to 0$ followed by $c\mapsto a$, $e\mapsto c$.  The second result is obtained by starting with \eqref{cdqHcdqH2} and taking the limit as $a\to 0$ followed by $c\mapsto a$, $d\mapsto c$. The third result is obtained by starting with \eqref{cdqHcdqH3} and taking the limit as $a\to 0$, $c\mapsto a$, $d\mapsto c$, $a\leftrightarrow b$.
\end{proof}

One can also obtain a connection formula between continuous big $q$-Hermite polynomials and Al-Salam--Chihara polynomials.

\begin{cor}
Let $n\in\N_0$, $q\in\CCdag$, $x=\frac12(z+z^{-1})\in\C$, $a,b,c\in\CCast$. Then
\begin{eqnarray}
&&\hspace{-1cm}H_n(x;a|q)=\frac{(bc,az^\pm;q)_\infty}{(ab,cz^\pm;q)_\infty}
\sum_{k=0}^\infty
\frac{a^{k-n}(\frac{c}{a};q)_k}{(q,bc;q)_k}
Q_k(x;b,c|q)
\qhyp32{q^{-n},\frac{qa}{c},ab}{q^{1-k}\frac{a}{c},0}{q,q}.
\end{eqnarray}
\end{cor}
\begin{proof}
Let $b\to 0$ in \eqref{ASCASC1} and then replace $d\mapsto b$ completes the proof.
\end{proof}
One can also obtain connection formulas between continuous big $q$-Hermite polynomials and continuous big $q$-Hermite polynomials.

\begin{cor}
Let $n\in\N_0$, $q\in\CCdag$, $x=\frac12(z+z^{-1})\in\C$, $a,b\in\CCast$. Then
\begin{eqnarray}
&&\hspace{-3.5cm}H_n(x;a|q)=
\frac{(ab;q)_\infty}{(bz^\pm;q)_\infty}\frac{1}{(ab;q)_n}
\sum_{k=0}^\infty \frac{q^{\binom{k}{2}}(-b)^{k}}{(q,q^nab;q)_k}Q_{n+k}(x;a,b|q).
\label{cbqHcbqH1}
\end{eqnarray}
\label{cor516}
\end{cor}
\begin{proof}
This result can be obtained by taking the limit $a\to 0$ in \eqref{ASCASC2} and then replacing $(b,c)\mapsto(a,b)$.
\end{proof}



\section{The Gupta--Masson rational $_4\phi_3$-functions}
\label{GuptaMasson}

From orthogonality relations for the $q^{-1}$-Al-Salam--Chihara polynomials it follows that the Poisson kernel in Theorem \ref{thm:PoissonKernel} in the case $ab=cd$ satisfies biorthogonality relations. This leads to another proof of the biorthogonality relation for rational $_4\phi_3$-functions in \cite[Corollary 4.4]{GuptaMasson1998},
see also \cite{GroeneveltWagenaar25} where this is proved using unitary maps. Let us briefly give a direct proof here.

Assume $0<q<1$, $ab=cd$, $ab>1$, $qb<a$, $qd<c$, so that the weight functions $W_m(a,b;q)$ and $W_m(c,d;q)$ for the $q^{-1}$-Al-Salam--Chihara polynomials, see \eqref{ASCqiO}, are positive. Using the shorthand notation 
\begin{equation}
P_{m,m'}^t = P_{m,m'}\left(\frac{qt}{ab};a,b,c,d\right),\label{shorth}
\end{equation}
the Poisson kernel from Theorem \ref{thm:PoissonKernel} can be written as
\[
P_{m,m'}^t= \sum_{n=0}^\infty h_n^t Q_n[aq^{-m};a,b|q^{-1}] Q_n[aq^{-m'};c,d|q^{-1}],
\]
with
\[
h_n^t = \frac{ q^{2\binom{n}{2}} }{(q,\frac{1}{ab};q)_n} \left( \frac{ qt}{ab} \right)^n.
\]
Let $|\frac{qb}{c}|<|t|<|\frac{c}{qb}|$, so that the sums for $P_{m,l}^{t}$ and $P_{m,l}^{t^{-1}}$ both converge absolutely. Now we have
\[
\begin{split}
	\sum_{l=0}^\infty W_{l}(c,d) P_{m,l}^{t} P_{m',l}^{t^{-1}} & = \sum_{n,n'=0}^\infty h_n^{t} h_{n'}^{t^{-1}} Q_n[aq^{-m};a,b|q^{-1}] Q_{n'}[aq^{-m'};a,b|q^{-1}] \\
    & \qquad \times \sum_{l=0}^\infty W_l(c,d) Q_n[cq^{-l};c,d|q^{-1}] Q_{n'}[cq^{-l};c,d|q^{-1}] \\
	& = \sum_{n=0}^\infty h_n^{t} h_{n}^{t^{-1}} Q_n[aq^{-m};a,b|q^{-1}] Q_{n}[aq^{-m'};a,b|q^{-1}] N_n(c,d;q),
\end{split}
\]
where we have used the orthogonality relations \eqref{ASCqi orth}. From the definitions of $N_n(a,b;q)$ and $h_n^t$ it follows that $h_n^{t} h_n^{t^{-1}} N_n(c,d;q) = (N_n(a,b;q))^{-1}$, so that applying the orthogonality relations dual to \eqref{ASCqi orth} gives biorthogonality relations for the Poisson kernel $P_{m,m'}$,
\begin{equation}
\sum_{l=0}^\infty W_{l}(c,d) P_{m,l}^{t} P_{m',l}^{t^{-1}} = \frac{ \delta_{m,m'} }{W_m(a,b)}.\label{pbio1}
\end{equation}
Using the symmetry from Remark \ref{rem32} it follows that the dual relations also hold,
\begin{equation}
\sum_{m=0}^\infty W_{m}(a,b) P_{m,l}^{t} P_{m,l'}^{t^{-1}} = \frac{ \delta_{l,l'} }{W_l(c,d)}.\label{pbio2}
\end{equation}
Explicitly, this gives the following biorthogonality relations. Under the parameter conditions given above, these are with respect to a positive weight function. The conditions can be removed using analytic continuation, leading to biorthogonality with respect to a complex-valued weight.

\begin{thm}Let $0<q<1$, $t,a,b,c,d\in\CCast$, such that $ab=cd$. Then
\begin{eqnarray}
&&\hspace{0.0cm}\sum_{l=0}^\infty \frac{ (1 -\frac{q^{2l}}{c^2})}{(1-\frac{1}{c^2})} \frac{ (\frac{qd}{c},\frac{1}{c^2};q)_l }{(q,\frac1{ab};q)_l } \left( \frac{c}{qd} \right)^l 
\qhyp43{q^{-l}, \frac{q^l}{c^2}, \frac{qat}{c}, \frac{qbt}{c} }{ \frac{q^{1-m}at}{c}, \frac{q^{m+1}t}{ac}, \frac{qd}{c} }{q,q}
\qhyp43{q^{-l}, \frac{q^l}{c^2}, \frac{qa}{ct}, \frac{qb}{ct} }{ \frac{q^{1-m'}a}{ct}, \frac{q^{m'+1}}{act}, \frac{qd}{c} }{q,q} \nonumber\\
&&  \hspace{4cm}= \frac{(\frac{q}{a^2},\frac{q}{c^2},\frac{qbt^\pm}{c};q)_\infty }{ (\frac{qb}{a},\frac{qd}{c}, \frac{qt^\pm}{ac};q)_\infty}\left( \frac{c}{qd} \right)^m \frac{(1-\frac{1}{a^2} )}{(1-\frac{q^{2m}}{a^2})} \frac{ (q,\frac{qb}{a},\frac{qt^\pm }{ac};q)_m }{( \frac{1}{a^2}, \frac{1}{ab},\frac{ct^\pm}{a};q)_m }\delta_{m,m'} ,\label{bio1}\\
&&\hspace{0.0cm}\sum_{m=0}^\infty \frac{( 1-\frac{q^{2m}}{a^2})}{(1-\frac{1}{a^2})} \frac{ (\frac{1}{a^2},\frac{1}{ab},\frac{ct^\pm}{a};q)_m }{(q,\frac{qb}{a},\frac{qt^\pm}{ac};q)_m } \left( \frac{qd}{c} \right)^m 
\qhyp43{q^{-l}, \frac{q^l}{c^2}, \frac{qat}{c}, \frac{qbt}{c} }{ \frac{q^{1-m}at}{c}, \frac{q^{m+1}t}{ac}, \frac{qd}{c} }{q,q}
\qhyp43{q^{-l'}, \frac{q^{l'}}{c^2}, \frac{qa}{ct}, \frac{qb}{ct} }{ \frac{q^{1-m}a}{ct}, \frac{q^{m+1}}{act}, \frac{qd}{c} }{q,q} \nonumber\\
&& \hspace{4.0cm} = \frac{(\frac{q}{a^2},\frac{q}{c^2},\frac{qbt^\pm}{c};q)_\infty }{ (\frac{qb}{a},\frac{qd}{c}, \frac{qt^\pm}{ac};q)_\infty}\left( \frac{qd}{c} \right)^l \frac{(1-\frac{1}{c^2})}{(1-\frac{q^{2l}}{c^2})} \frac{ (q,\frac{1}{cd};q)_l }{( \frac{1}{c^2}, \frac{qd}{c};q)_l }\delta_{l,l'},\label{bio2}
\end{eqnarray}
where $m,m'\in\N_0$, $|qb|<|c|$ in \eqref{bio1} and $l,l'\in\N_0$, $|qd|<|c|$ in \eqref{bio2}.
\end{thm}
\begin{proof}
These biorthogonality relations are obtained by using \eqref{pbio1}, \eqref{pbio2} and substituting these using the definitions of the weight function $W$ and the Poisson kernel $P_{m,m'}$ \eqref{shorth}. The conditions on convergence are obtained by performing an asymptotic analysis as the sum indices $l,m$ become very large respectively. For \eqref{bio1}, \eqref{bio2} write the summands on the left-hand sides as $A_{l,m,m'}(t;a,b,c,d;q)$, $B_{m,l,l'}(t;a,b,c,d;q)$
respectively. Then it is straightforward to show using the Sears' balanced transformation \cite[\href{http://dlmf.nist.gov/17.9.E14}{(17.9.14)}]{NIST:DLMF} that one has the following asymptotic approximations, namely 
\begin{eqnarray}
&&\hspace{-4cm}A_{l,m,m'}(t;a,b,c,d;q)\sim a_{m,m'}(t;a,b,c,d;q)\left(\frac{qb}{a}\right)^l, \mbox{ as } l\to\infty\\
&&\hspace{-4cm}B_{m,l,l'}(t;a,b,c,d;q)\sim b_{l,l'}(t;a,b,c,d;q)\left(\frac{qd}{c}\right)^m\mbox{ as } m\to\infty,
\end{eqnarray}
where $a_{m,m'}$ is independent of $l$ and $b_{l,l'}$ is independent of $m$. This completes the proof.
\end{proof}

Using Theorem \ref{Thm:sum AWpol} and orthogonality relations for Askey--Wilson polynomials, we can express the biorthogonal rational $_4\phi_3$-functions in terms of Askey--Wilson polynomials. The well-known orthogonality relation for the Askey--Wilson polynomials is given by \cite[(14.1.2)]{Koekoeketal}
\begin{eqnarray}
&&\hspace{-3.5cm}
\int_{-1}^1 p_n(x;a,b,c,d|q) p_{n'}(x;a,b,c,d|q) w(x;a,b,c,d|q)\, \frac{dx}{\sqrt{1-x^2}} \nonumber\\ && \hspace{-2cm}=  \frac{2\pi (q^{n-1}abcd;q)_n (q^{2n}abcd;q)_\infty}{(q^{n+1},q^nab,q^nac,q^nad,q^nbc,q^nbd,q^ncd;q)_\infty}\delta_{n,n'},\label{AWorthog}
\end{eqnarray}
where $0<q<1$, the parameters $a,b,c,d$ are real or appear in pairs of complex conjugates and $|a|,|b|,|c|,|d|<1$, and
\begin{eqnarray*}
&&\hspace{-3cm}w(x;a,b,c,d|q) := \frac{ (z^{\pm 2};q)_\infty}{(az^\pm,bz^\pm,cz^\pm, dz^\pm;q)_\infty}, \quad x = \frac12(z+z^{-1}).
\end{eqnarray*}
\medskip

\noindent Using \eqref{AWorthog} with Theorem \ref{Thm:sum AWpol} we obtain the following expressions for the rational $_4\phi_3$-function, where we have chosen to write the Askey--Wilson polynomials in their `natural' parameters, so that the parameters $a,b,c,d$ below are different from the parameters in the biorthogonality relations above.
\begin{cor}
\label{cor72}
Let $n,k\in\N_0$, $0<q<1$, $a,b,c,d,e,f\in\CCast$ such that $|a|,|b|,|c|,|f|<1$. Then  
\begin{eqnarray}
&&\hspace{-0.7cm}\qhyp43{q^{-n},q^{n-1}abcd,cf,\frac{qc}{e}}{cd,q^kabcf,q^{1-k}\frac{c}{e}}{q,q} = c^{n-k}\frac{(q,cf,q^nac,q^nbc,q^kab,q^kaf,q^kbf;q)_\infty}{2\pi(cd;q)_n(\frac{e}{c};q)_k ({q^k}abcf;q)_\infty}\,{\mathcal I}_{n,k}(a,b,c,d,e,f|q),\label{rat43a}  \\
&&\hspace{-0.7cm}\qhyp43{q^{-k},q^{k-1}abef,cf,\frac{qf}{d}}{ef,q^nabcf,q^{1-n}\frac{f}{d}}{q,q} = f^{k-n}\frac{(q,cf,q^nab,q^nac,q^nbc,q^kaf,q^kbf;q)_\infty}{2\pi(\frac{d}{f};q)_n(ef;q)_k (q^nabcf;q)_\infty}\,{\mathcal I}_{n,k}(a,b,c,d,e,f|q),\label{rat43b} 
\end{eqnarray}
where
\begin{eqnarray}
{\mathcal I}_{n,k}(a,b,c,d,e,f|q):=\int_{-1}^1 p_n(x;a,b,c,d|q) p_k(x;a,b,e,f|q)  w(x;a,b,c,f|q)\,\frac{dx}{\sqrt{1-x^2}}.
\end{eqnarray}
\label{cor62}
\end{cor}
\begin{proof}
In order to write the Gupta--Masson   rational $_4\phi_3$-functions in terms of an integral over Askey--Wilson polynomials, assume $a,b,c,e,f \in (-1,1)$. 
Then multiply the first identity in Theorem \ref{Thm:sum AWpol} by
\[
p_k(x;a,b,e,f|q) \frac{w(x;a,b,c,f)}{\sqrt{1-x^2}} = p_k(x;a,b,e,f|q) \frac{(ez^\pm;q)_\infty w(x;a,b,e,f|q)}{(cz^{\pm};q)_\infty\sqrt{1-x^2}},
\]
integrate over $x\in(-1,1)$ and use the Askey--Wilson orthogonality relation \eqref{AWorthog}. This produces \eqref{rat43a}. The proof of the relation \eqref{rat43b} is similar except instead applied to \eqref{thmGAWeq2}. The restrictions on the parameters can be removed afterwards by analytic continuation.
\end{proof}


\begin{rem}
Note that Corollary \ref{cor62} can be considered as a generalization of the Askey--Wilson orthogonality relation \eqref{AWorthog}. Clearly, for $(e,f) \mapsto  (c,d)$ the integral ${\mathcal I}_{n,k}$ reduces to the integral in \eqref{AWorthog}. In this case, the ${}_4\phi_3$ rational functions become ${}_3\phi_2$'s which can be evaluated using 
the $q$-Pfaff--Saalsch\"utz sum
\cite[\href{http://dlmf.nist.gov/17.7.E4}{(17.7.4)}]{NIST:DLMF}.
\end{rem}

\noindent For \eqref{rat43a}, one has the following specializations and limits. For $f=d$ the left hand side becomes a $q$-Hahn type rational function, see e.g.,~\cite{Bussiereetal2022}. 
Furthermore, letting $f\to 0$ or $d \to 0$, we find an identity expressing a $q$-Hahn type rational function in terms of an integral of a product of a continuous dual $q$-Hahn polynomial and an Askey--Wilson polynomial. Letting $e,f \to 0$ the $_4\phi_3$ on the left-hand side becomes a $_2\phi_1$-function, which is essentially a $q^{-1}$-Al-Salam--Chihara polynomial, leading to an identity which is equivalent to Theorem \ref{thm:Groenevelt}.

Finally, let us remark that by writing an Askey--Wilson polynomial as a $_4\phi_3$-series using \eqref{aw:def1} and using the symmetry $p_n(x;a,b,c,d|q) = p_n(x;c,d,a,b|q)$ it follows that the $_4\phi_3$ rational function can be written in terms of an Askey--Wilson polynomial as
\begin{equation} \label{eq:rat4phi3=AW}
\begin{split}
&\qhyp43{q^{-n},q^{n-1}abcd,cf,\frac{qc}{e}}{cd,q^m abcf,q^{1-m}\frac{c}{e}}{q,q}  \\
&\quad =\frac{\left( c \sqrt{\tfrac{qf}{e}}\right)^{n}}{(cd,q^m abcf, q^{1-m}\tfrac{c}{e};q)_n }\,p_n\left[\sqrt{\tfrac{ef}{q}};q^{m-\frac12}ab\sqrt{ef}, \tfrac{q^{\frac12-m}}{\sqrt{ef}}, c\sqrt{\tfrac{qf}{e}}, d\sqrt{\tfrac{e}{qf}}\, | q\right].
\end{split}
\end{equation}
Using this identity one can rewrite identities for Askey--Wilson polynomials as identities for the $_4\phi_3$ rational functions. As an example, the connection relation \eqref{AW1free} with one free parameter can be turned into an connection formula for the $_4\phi_3$ rational functions.
\begin{cor}
Let $m,n \in \N_0$, $0<q<1$, $a,b,c,d,e,f,g \in \CCast$, then
\begin{eqnarray}
&&\hspace{-0.7cm}\qhyp43{q^{-n},q^{n-1}abcd,cf,\frac{qc}{e}}{cd,q^m abcf,q^{1-m}\frac{c}{e}}{q,q}  =\left(\frac{qcfg}{e}\right)^n \frac{(q,ab;q)_n}{(cd,\frac{qabcfg}{e};q)_n} \nonumber\\
&&\hspace{-0.5cm}\times \sum_{k=0}^n \frac{ (\frac{de}{qfg};q)_{n-k}(\frac{abcfg}{e},q^{n-1}abcd,q^mabfg;q)_k (\frac{qabcfg}{e};q)_{2k}}{(q^mabfg)^k(q;q)_{n-k}(q,\frac{q^{1-m}c}{e},\frac{q^{n+1}abcfg}{e};q)_k} \qhyp43{q^{-k},\frac{q^kabcfg}{e},q^mab,q^{m-1}abef}{ab,q^mabcf,q^mabfg}{q,q}.
\end{eqnarray}
\label{cor74}
\end{cor}
\begin{proof}
Using \eqref{AW1free} with $x=\frac12(\sqrt{\tfrac{ef}{q}}+\sqrt{\tfrac{q}{ef}})$ and parameters
\[
(a,b,c,d,e) \mapsto \left(q^{m-\frac12}ab\sqrt{ef}, \tfrac{q^{\frac12-m}}{\sqrt{ef}}, c\sqrt{\tfrac{qf}{e}}, d\sqrt{\tfrac{qf}{e}}, g\sqrt{\tfrac{qf}{e}}\right),
\]
identity \eqref{eq:rat4phi3=AW} and 
\begin{eqnarray}
&&\hspace{0.0cm}p_k\left[\sqrt{\frac{ef}{q}};q^{m-\frac12}ab\sqrt{ef},\frac{q^{\frac12-m}}{\sqrt{ef}},c\sqrt{\frac{qf}{e}},g\sqrt{\frac{qf}{e}}|q\right]\nonumber\\
&&\hspace{1.0cm}=\left(q^{m-\frac12}ab\sqrt{ef}\right)^{-k}(ab,q^mabcf,q^mabfg;q)_k\qhyp43{q^{-k},\frac{q^kabcfg}{e},q^mab,q^{m-1}abef}{ab,q^mabcf,q^mabfg}{q,q},
\end{eqnarray}
completes the proof.
\end{proof}
}


\appendix

\bibliographystyle{plain}

\begin{thebibliography}{10}

\bibitem{AskeyWilson85}
R.~Askey and J.~Wilson.
\newblock Some basic hypergeometric orthogonal polynomials that generalize
  {J}acobi polynomials.
\newblock {\em Memoirs of the American Mathematical Society}, 54(319):iv+55,
  1985.

\bibitem{Bussiereetal2022}
I.~Bussi\`ere, J.~Gaboriaud, L.~Vinet, and A.~Zhedanov.
\newblock Bispectrality and biorthogonality of the rational functions of
  {$q$}-{H}ahn type.
\newblock {\em Journal of Mathematical Analysis and Applications}, 516(1):Paper
  No. 126443, 17, 2022.

\bibitem{CohlCostasSantos20b}
H.~S. {Cohl} and R.~S. {Costas-Santos}.
\newblock {Symmetry of terminating basic hypergeometric representations of the
  Askey-Wilson polynomials}.
\newblock {\em Journal of Mathematical Analysis and Applications},
  517(1):126583, 2023.

\bibitem{NIST:DLMF}
{\it NIST Digital Library of Mathematical Functions}.
\newblock \href{https://dlmf.nist.gov/}{{\bf\tt\normalsize
  https://dlmf.nist.gov/}}, Release 1.2.5 of 2025-12-15.
\newblock F.~W.~J. Olver, A.~B. {Olde Daalhuis}, D.~W. Lozier, B.~I. Schneider,
  R.~F. Boisvert, C.~W. Clark, B.~R. Miller, B.~V. Saunders, H.~S. Cohl, and
  M.~A. McClain, eds.

\bibitem{GaspRah}
G.~Gasper and M.~Rahman.
\newblock {\em Basic hypergeometric series}, volume~96 of {\em Encyclopedia of
  Mathematics and its Applications}.
\newblock Cambridge University Press, Cambridge, second edition, 2004.
\newblock With a foreword by Richard Askey.

\bibitem{Groenevelt2004}
W.~Groenevelt.
\newblock Bilinear summation formulas from quantum algebra representations.
\newblock {\em The Ramanujan Journal}, 8(3):383--416, 2004.

\bibitem{Groenevelt2021}
W.~Groenevelt.
\newblock A quantum algebra approach to multivariate {A}skey-{W}ilson
  polynomials.
\newblock {\em International Mathematics Research Notices. IMRN},
  2021(5):3224--3266, 2021.

\bibitem{GroeneveltWagenaar25}
W.~{Groenevelt} and C.~{Wagenaar}.
\newblock {Quantum algebra approach to univariate and multivariate rational
  functions of $q$-Racah type}.
\newblock {\em arXiv:2507.13483}, page 27 pages, 2025.

\bibitem{GuptaMasson1998}
D.~P. Gupta and D.~R. Masson.
\newblock Contiguous relations, continued fractions and orthogonality.
\newblock {\em Transactions of the American Mathematical Society},
  350(2):769--808, 1998.

\bibitem{Ismail}
M.~E.~H. Ismail.
\newblock {\em Classical and {Q}uantum {O}rthogonal {P}olynomials in {O}ne
  {V}ariable}, volume~98 of {\em {Encyclopedia of Mathematics and its
  Applications}}.
\newblock Cambridge University Press, Cambridge, 2005.
\newblock With two chapters by Walter Van Assche.

\bibitem{IsmailZhang2005}
M.~E.~H. Ismail and R.~Zhang.
\newblock New proofs of some {$q$}-series results.
\newblock In {\em Theory and applications of special functions}, volume~13 of
  {\em Developments in Mathematics}, pages 285--299. Springer, New York, 2005.

\bibitem{Koekoeketal}
R.~Koekoek, P.~A. Lesky, and R.~F. Swarttouw.
\newblock {\em Hypergeometric orthogonal polynomials and their
  {$q$}-analogues}.
\newblock Springer Monographs in Mathematics. Springer-Verlag, Berlin, 2010.
\newblock With a foreword by Tom H. Koornwinder.

\bibitem{RosengrenCONM2000}
H.~Rosengren.
\newblock A new quantum algebraic interpretation of the {A}skey-{W}ilson
  polynomials.
\newblock In {\em {$q$}-series from a contemporary perspective ({S}outh
  {H}adley, {MA}, 1998)}, volume 254 of {\em Contemporary Mathematics}, pages
  371--394. Amer. Math. Soc., Providence, 2000.

\bibitem{vandeBultRains09}
F.~J. van~de Bult and E.~M. Rains.
\newblock Basic hypergeometric functions as limits of elliptic hypergeometric
  functions.
\newblock {\em Symmetry, Integrability and Geometry: Methods and Applications},
  5(059), 2009.

\end{thebibliography}

\def\cprime{$'$} \def\dbar{\leavevmode\hbox to 0pt{\hskip.2ex \accent"16\hss}d}

\end{document}